\documentclass[11pt]{amsart}

\usepackage[utf8]{inputenc}
\usepackage[T1]{fontenc}

\usepackage{amsmath,amssymb,amsfonts,textcomp,amsthm,xifthen,graphicx,color,pgfplots}

\usepackage{enumerate}
\usepackage{enumitem}

\usepackage{multirow}

\usepackage{fullpage}

\usepackage[utf8]{inputenc}

\usepackage[pdftex,
            pdfauthor={Thomas F\"uhrer and Norbert Heuer},
            pdftitle={Robust DPG Fortin operators},
            ]{hyperref}

\usepackage{pgfplots}
\usepgfplotslibrary{groupplots}
\pgfplotsset{compat=newest}
\usepgfplotslibrary{external}
\usepgfplotslibrary{colorbrewer}
\newtheorem{theorem}{Theorem}
\newtheorem{lemma}[theorem]{Lemma}
\newtheorem{corollary}[theorem]{Corollary}
\newtheorem{proposition}[theorem]{Proposition}
\newtheorem{remark}[theorem]{Remark}

\newcommand{\estDPG}{\mathrm{est}}

\newcommand{\Tref}{{\widehat T}}

\newcommand{\PiF}{\Pi_{\mathrm{F}}}
\newcommand{\PiFtilde}{\widetilde{\Pi}_{\mathrm{F}}}
\newcommand{\PiFa}{\Pi_{\mathrm{F},\alpha}}
\newcommand{\PiFatilde}{\widetilde{\Pi}_{\mathrm{F},\alpha}}

\newcommand{\PiFptilde}{\widetilde{\Pi}_{\mathrm{F},hp}}
\newcommand{\PiFaptilde}{\widetilde{\Pi}_{\mathrm{F},hp,\alpha}}

\newcommand{\PiFp}{\Pi_{\mathrm{F},hp}}
\newcommand{\PiFap}{\Pi_{\mathrm{F},hp,\alpha}}

\newcommand{\conv}{\operatorname{conv}}

\newcommand{\aalpha}{\boldsymbol{\alpha}}

\newcommand\osc{\mathrm{osc}}

\DeclareMathOperator{\linhull}{span}

\newcommand{\ip}[2]{(#1\hspace*{.5mm},#2)}
\newcommand{\dual}[2]{\langle#1\hspace*{.5mm},#2\rangle}

\newcommand{\norm}[3][]{#1\|#2#1\|_{#3}}

\newcommand{\diam}{\mathrm{diam}}
\newcommand{\di}{\mathrm{d}}

\newcommand\ccurl{{\bf curl\,}}
\renewcommand\div{\operatorname{div}}

\newcommand{\Hdivset}[1]{\boldsymbol{H}(\div;#1)}

\newcommand{\set}[2]{\big\{#1\,:\,#2\big\}}

\newcommand{\RT}{\ensuremath{\boldsymbol{RT}}}

\newcommand{\R}{\ensuremath{\mathbb{R}}}
\newcommand{\N}{\ensuremath{\mathbb{N}}}
\newcommand{\HH}{\ensuremath{{\boldsymbol{H}}}}

\newcommand{\LL}{\ensuremath{\boldsymbol{L}}}

\newcommand{\vv}{\ensuremath{\boldsymbol{v}}}

\newcommand{\TT}{\ensuremath{\mathcal{T}}}
\newcommand{\FF}{\ensuremath{\mathcal{F}}}

\newcommand{\PP}{\ensuremath{\boldsymbol{P}}}

\newcommand{\OO}{\ensuremath{\mathcal{O}}}
\newcommand{\EE}{\ensuremath{\mathcal{E}}}

\newcommand{\normal}{\ensuremath{{\boldsymbol{n}}}}
\newcommand{\tangential}{\ensuremath{{\boldsymbol{t}}}}

\newcommand{\VV}{\ensuremath{\mathcal{V}}}

\newcommand\eeta{\boldsymbol{\eta}}

\newcommand{\ppsi}{{\boldsymbol\psi}}

\newcommand{\ssigma}{{\boldsymbol\sigma}}
\newcommand{\ttau}{{\boldsymbol\tau}}

\newcommand{\qq}{{\boldsymbol{q}}}
\newcommand{\uu}{\boldsymbol{u}}

\newcommand{\tr}{\operatorname{tr}}

\newcommand{\trdiv}{\tr^{\div}}
\newcommand{\trgrad}{\tr^{\nabla}}

\newcommand{\xx}{{\boldsymbol{x}}}

\newcommand{\est}{\operatorname{est}}

\begin{document}

\title{Robust DPG Fortin operators}
\date{\today}

\author{Thomas F\"uhrer}
\author{Norbert Heuer}
\address{Facultad de Matem\'{a}ticas, Pontificia Universidad Cat\'{o}lica de Chile, Santiago, Chile}
\email{\{tofuhrer,nheuer\}@mat.uc.cl}

\thanks{{\bf Acknowledgment.} 
This work was supported by ANID through FONDECYT projects 1210391 and 1230013.}

\begin{abstract}
At the fully discrete setting, stability of the discontinuous Petrov--Galerkin (DPG) method with optimal test functions
requires local test spaces that ensure the existence of Fortin operators.
We construct such operators for $H^1$ and $\HH(\mathrm{div})$ on simplices in any space dimension
and arbitrary polynomial degree. The resulting test spaces are smaller than previously analyzed cases.
For parameter-dependent norms, we achieve uniform boundedness by the inclusion of exponential layers.
As an example, we consider a canonical DPG setting for reaction-dominated diffusion.
Our test spaces guarantee uniform stability and quasi-optimal convergence of the scheme.
We present numerical experiments that illustrate the loss of stability and error control by the residual for small diffusion coefficient when using
standard polynomial test spaces, whereas we observe uniform stability and error control with our construction.
\end{abstract}
\maketitle


\section{Introduction}\label{sec:intro}
Fortin operators are a critical tool for the stability analysis of mixed finite element schemes, cf.~\cite{BoffiBrezziFortin}.
The discontinuous Petrov--Galerkin (DPG) method with optimal test functions, on the other hand, is a framework that aims at automatic inf-sup stability. In practice, optimal test functions have to be approximated and the question of existence of Fortin operators re-appears. In this case, local (element-wise) operators are sufficient. A first answer was given in~\cite{practicalDPG}, with subsequent studies in~\cite{breakSpace,constrFortin,DemkowiczZanotti20,KLove2}.
In the case of singularly-perturbed problems, uniform discrete stability, or robustness of the method, requires the existence of uniformly bounded Fortin operators. This has been an open problem. 
In this paper, we present local Fortin operators for $H^1$ and $\HH(\mathrm{div})$, on simplices in arbitrary dimension and arbitrary polynomial degree. In contrast to previous results, our constructions are explicit (not needed in applications) and require fewer degrees of freedom. More importantly, we include parameter-dependent exponential layers that guarantee uniform boundedness of our operators for parameter-dependent norms (the $\HH(\mathrm{div})$-case is restricted to two and three space dimensions). We illustrate their application to a DPG method for a reaction-dominated diffusion problem, leading to robustness, i.e., uniform stability, error control, and convergence. In this case, we consider the energy-norm induced by the problem. We have not analyzed the case of balanced norms as proposed in~\cite{DPGrefusion}. This and possible extensions to other singularly-perturbed problems like advection-dominated diffusion are left to future research.

Let us shortly discuss the abstract setting of the DPG method: Consider the variational formulation 
\begin{align*}
  u\in U\colon \qquad b(u,v) = L(v) \quad\forall v\in V, 
\end{align*}
where $U$, $V$ are Hilbert spaces with norms $\norm\cdot{U}$, $\norm{\cdot}V$, $b(\cdot,\cdot)$ is a bounded bilinear form and induces a boundedly invertible operator $B\colon U\to V'$, $u\mapsto b(u,\cdot)$.
Choosing finite dimensional spaces $U_h\subset U$, $V_h\subset V$, the fully discrete DPG method reads:
\begin{align}\label{eq:DPG:practical}
  u_h \in U_h\colon \qquad b(u_h,v) = L(v) \quad\forall v\in \Theta_h(U_h), 
\end{align}
where $\Theta_h\colon U\to V_h$ is defined through ($\ip\cdot\cdot_V$ being the inner product on $V$)
\begin{align*}
  \ip{\Theta_h u}{v_h}_V = b(u,v_h) \quad\forall v_h \in V_h. 
\end{align*}
Well-posedness of discrete DPG is ensured if there exists a Fortin operator $\PiF\colon V\to V_h$ such that
\begin{align}\label{eq:fortin:abstract}
  \norm{\PiF v}V &\leq C_{\PiF} \norm{v}V, \quad
  b(u_h,v-\PiF v) = 0 \quad\forall u_h\in U_h, \, v\in V,
\end{align}
see,~\cite{practicalDPG}. Furthermore, the existence of a Fortin operator also implies quasi-optimality,
\begin{align*}
  \norm{u-u_h}U \leq C C_{\PiF} \min_{w_h\in U_h} \norm{u-w_h}U
\end{align*}
with $C=C_b/c_b$ where $C_b$ and $c_b$ are the boundedness and $\inf$-$\sup$ constants of $b(\cdot,\cdot)$, respectively.
It also plays an important role in the a posteriori error control, see~\cite[Theorem~2.1]{DPGaposteriori}, 
\begin{align*}
  \norm{u-u_h}U \eqsim \norm{Bu_h-L}{V'} + \osc(L), 
\end{align*}
where $\osc(L) = \sup_{0\neq v\in V} L(v-\PiF v)/\norm{v}V$.

One of the main motivations that had driven the development of the DPG method was to derive robust numerical schemes for singularly perturbed problems, see, e.g.,~\cite{DPGconfusion,DPGrefusion}.
All these problems have in common that they naturally lead to parameter dependent trial and test norms. For a concrete example consider the test space $H^1(T)$ (here $T$ as an element of the mesh $\TT$) equipped with the norm
\begin{align*}
  \norm{v}{T,\alpha} = \norm{v}T + \alpha \norm{\nabla v}T \quad\text{for some fixed } \alpha>0. 
\end{align*}
In~\cite{practicalDPG,breakSpace,DemkowiczZanotti20} Fortin operators are constructed (resp. their existence is shown). 
Let us write $\PiF\colon H^1(T)\to P^k(T)$ (a polynomial space). Besides some conditions they satisfy the boundedness estimates
\begin{align*}
  \norm{\PiF v}T \lesssim \norm{v}T + h_T\norm{\nabla v}T, \quad\norm{\nabla \PiF v}T \lesssim \norm{\nabla v}T. 
\end{align*}
Combining the latter two estimates yields
\begin{align*}
  \norm{\PiF v}{T,\alpha} \lesssim \max\{1,\alpha^{-1}h_T\} \norm{v}{T,\alpha}.
\end{align*}
If $h_T\lesssim \alpha$ it is clear that $\norm{\PiF v}{T,\alpha} \lesssim \norm{v}{T,\alpha}$ uniformly. 
However, when $\alpha\lesssim h_T$ --- a case often encountered with singularly perturbed problems --- then we get $\norm{\PiF v}{T,\alpha} \lesssim \alpha^{-1}h_T \norm{v}{T,\alpha}$. This means that, particularly on coarse meshes, the Fortin operator is not uniformly bounded. By our prior considerations, this means that quasi-optimality as well as a posteriori error control is spoiled. 

One of the main objectives of this work is to define discrete spaces $V_h$ and construct corresponding Fortin operators $\PiF\colon V\to V_h$ with boundedness constant $C_{\PiF}\eqsim 1$ for small parameters ($\alpha\lesssim h_T$ in the previous example). 
We do this by first revisiting the construction of Fortin operators for the spaces $H^1(T)$ and $\Hdivset{T}$ in the case $h_T\lesssim \alpha$. 
Contrary to prior works we construct our Fortin operators in an explicit manner. This allows to precisely write down a basis for the discrete test spaces yielding --- compared with the operators from~\cite{practicalDPG,breakSpace,DemkowiczZanotti20} --- smaller dimensions.
The novel idea of definition also requires a different analysis which we present in detail. 
Our constructions are valid for any polynomial degree (this statement will be made precise below) and arbitrary dimension (except for some operators from Section~\ref{sec:fortinDiv}).
However, main advantage is that the definition and construction can be extended to the case $\alpha\lesssim h_T$. 
Specifically, we use modified face bubble functions $\eta_{\alpha,F}$ instead of polynomial face bubble functions $\eta_F$. They are defined in such a way that their volume norm $\norm{\eta_{\alpha,F}}T$ resp. $\norm{\nabla \eta_{\alpha,F}}T$ scales differently than $\norm{\eta_F}T$ resp. $\norm{\nabla \eta_F}T$ depending on the ratio $\alpha/h_T$, though $\eta_{\alpha,F}|_{\partial T} = \eta_F|_{\partial T}$.
The analysis requires some additional tools and steps. 
We also consider low order polynomial cases which allow for even smaller dimensions in the test space. 

As mentioned above, Fortin operators for the DPG method have been constructed in various works: 
The first one was~\cite{practicalDPG}. Other articles that analyze the existence of Fortin operators for second-order PDEs include~\cite{breakSpace,constrFortin,DemkowiczZanotti20}.
The latter references consider all arbitrary fixed polynomial orders. For low order methods with smaller test space dimensions we refer to~\cite{CarstensenGHW14,CarstensenHellwig16}.
For a fourth-order PDE model problem we have shown existence of Fortin operators in~\cite{KLove2}.

The remainder of this work is organized as follows: In Section~\ref{sec:basis} we introduce some notation and define basis functions as well as novel face bubble functions. 
Section~\ref{sec:fortin} and Section~\ref{sec:fortinDiv} discuss the construction of Fortin operators for $H^1$ and $\HH(\div)$, respectively. 
Section~\ref{sec:num} concludes this article with a short description of a DPG method for a singularly perturbed reaction-diffusion problem and numerical examples. 

\section{Preliminaries}\label{sec:basis}
The notation $a\lesssim b$ ($a\gtrsim b$) for $a,b>0$ means that there exists $C>0$ such that $a\leq C\,b$ ($C\,a\geq b$). We write $a\eqsim b$ for $a,b>0$ if $a\lesssim b \lesssim a$.
The generic constant $C$ is independent of involved functions, the diameter of elements, and parameters like $\alpha$ and $\varepsilon$, where present. 

\subsection{Mesh and spaces}
Let $\TT$ denote a shape-regular simplicial mesh of a Lipschitz domain $\Omega$ with $\diam(\Omega)\leq 1$. 
Throughout, $T\in\TT$ is some fixed element, $\Tref$ is the reference element given as the convex hull of the origin and the $n$ coordinate axis vectors. E.g., for $n=2,3$ it reads
\begin{align*}
  \Tref = \begin{cases}
    \conv\{(0,0)^\top,(1,0)^\top,(0,1)^\top\}, & n=2,\\
    \conv\{(0,0,0)^\top,(1,0,0)^\top,(0,1,0)^\top,(0,0,1)^\top\}, & n=3.
  \end{cases} 
\end{align*}
Here, we understand $\conv$ as the interior of the convex hull of a set.

We adopt the standard notation for Lebesgue and Sobolev spaces, $L^2(\omega)$, $\LL^2(\omega) := L^2(\omega;\R^n)$, 
\begin{align*}
  H^1(\omega) = \set{v\in L^2(\omega)}{\nabla v\in \LL^2(\omega)}, 
  \quad
  \Hdivset\omega = \set{\ttau\in \LL^2(\omega)}{\div\ttau\in L^2(\omega)}
\end{align*}
for a Lipschitz domain $\omega\subset \R^n$.
With $\normal_\omega$ we denote the normal vector on the boundary of $\omega$ pointing from $\omega$ to its complement $\R^n\setminus\omega$.
Recall that traces of $H^1(\omega)$ elements are well defined (in the sense of trace operators) and the canonic trace space is $H^{1/2}(\partial\omega)$.
Normal traces of $\Hdivset\omega$ elements are well defined (in a duality sense) and the canonic trace space is $H^{-1/2}(\partial\omega)$. We simply write $\ttau\cdot\normal_\omega$ for the normal trace of $\ttau\in \Hdivset\omega$.

We denote by $\|\cdot\|_\omega$ the canonical $L^2(\omega)$ norm induced by the $L^2(\omega)$ inner product $\ip{\cdot}{\cdot}_\omega$ for a Lipschitz domain $\omega\subset \R^n$. The volume measure of $\omega$ is given by $|\omega|$. The same notation for norm and inner product is used for $L^2(\omega;\R^n)$.
The surface measure of $\gamma\subseteq \partial \omega$ is denoted by $|\gamma|$ and $\|v\|_\gamma$ is the $L^2(\gamma)$ norm induced by the inner product $\dual{\cdot}{\cdot}_\gamma$. We also use the same notation for the duality between $H^{-1/2}(\partial \omega)$ and $H^{1/2}(\partial \omega)$, 
\begin{align*}
  \dual{\phi}{v}_{\partial \omega} \quad\text{for } \phi \in H^{-1/2}(\partial \omega), v\in H^{1/2}(\partial \omega). 
\end{align*}
Recall the following relation between traces of $H^1(\omega)$ and $\Hdivset\omega$, 
\begin{align}\label{eq:dualityrelation}
  \dual{\ttau\cdot\normal_\omega}{v}_{\partial \omega} = \ip{\div\ttau}v_\omega + \ip{\ttau}{\nabla v}_\omega
\end{align}
for all $\ttau\in \Hdivset\omega$, $v\in H^1(\omega)$. Obviously, for sufficiently regular functions this is just the integration by parts formula.

Let $\Pi_T^q\colon L^2(T)\to P^q(T)$ denote the $L^2(T)$ orthogonal projection on $P^q(T)$, the space of polynomials on $T$ of degree less than or equal to $q\in\N_0$. 
For vector-valued polynomials (each component is a polynomial of degree less than or equal to $q$) we use the symbol $\PP^q(T)$.
Recall the first-order approximation property
\begin{align*}
  \|v-\Pi_T^qv\|_T\leq\|v-\Pi_T^0v\|_T \lesssim h_T \|\nabla v\|_T \quad\text{for all }v\in H^1(T).
\end{align*}
An important tool is the following (multiplicative version) of the trace inequality. It can be derived from~\cite[Theorem~1.6.6]{BrennerScott08} with a scaling argument (using the reference element $\Tref$) and the approximation property of $\Pi_T^0$.
\begin{lemma}\label{lem:traceineq}
  For any $v\in H^1(T)$ we have
  \begin{align}\label{eq:traceineq}
    \norm{v-\Pi_T^0v}{\partial T} \lesssim \norm{v-\Pi_T^0v}T^{1/2}\norm{\nabla v}{T}^{1/2} \lesssim h_T^{1/2}\norm{\nabla v}T
  \end{align}
  with hidden constants only depending on the shape of $T$.
\end{lemma}

We denote by $\VV_T$ the set of the $n+1$ vertices of $T$, $\FF_T$ is the set of $n+1$ faces of $T$ and $\VV_F$ denotes the set of $n$ vertices of $F$.
For $z\in \VV_T$ let $F_z\in\FF_T$ be the face opposite to $z$, i.e., $F_z = \conv\big(\VV_T\setminus \{z\}\big)$.
Similarly, for $F\in\FF_T$ let $z_F\in\VV_T$ be the vertex opposite to $F$.
For $F\in\FF_T$ let $P^q(\FF_T)\subset L^2(\partial T)$ denote face-wise polynomials of degree less than or equal to $q\in\N_0$ and $P_c^q(\FF_T):=P^q(\FF_T)\cap C^0(\partial T)$. 
Note that $\normal_T$ is face-wise constant. For the fixed element $T\in\TT$ and any $F\in\FF_T$ we abbreviate $\normal_F = \normal_T|_F$.
For a vertex $z\in\VV_T$ we denote by $\EE_z$ the set of $n$ edges that share the same vertex $z$, i.e., for each $E\in\EE_z$ there is a $z'\in\VV_T\setminus\{z\}$ with $E = \conv\{z,z'\}$. To each $E=\conv\{z,z'\}\in\EE_z$ we associate the (tangential) vector $\tangential_E = z'-z$ (the orientation does not matter for our analysis nor for implementation).

Furthermore, $h_T = \diam(T)$, $h_T \eqsim \diam(F)$ for all $F\in\FF_T$ with hidden constants only depending on the shape of $T$.
Some other relations that we frequently use without further notice are $|T|\eqsim h_T^n$, $|F|\eqsim h_T^{n-1}$, $|\partial T|\eqsim |F|$, $|T|\eqsim |F|h_T$ for any $F\in\FF_T$.

In the following subsections we define special functions that will be used for the construction of the Fortin operators. For ease of reading and reference they are listed, together with their relevant properties, in Table~\ref{tab:basis} below.

\subsection{Low-order basis functions}\label{sec:low}
The functions $\eta_z\in P^1(T)$ are canonical basis functions with $\eta_z(z') = \delta_{z,z'}$ for $z,z'\in \VV_T$ and $\delta_{\cdot,\cdot}$ denoting the Kronecker-$\delta$,
\begin{align*}
  \eta_F = \prod_{z\in\VV_F} \eta_z \in P^{n}(T)
\end{align*}
are the face bubble functions, and $\eta_{T} = \prod_{z\in\VV_T} \eta_{z}\in P^{n+1}(T)$ is the element bubble function.
Clearly, $\linhull\{\eta_z\,:\,z\in \VV_T\}= P^1(T)$.
Alternatively, we may use the following basis for $P^1(T)$: 
First, we abbreviate $d_F = \eta_{z_F}$. Then, for $F\in\FF_T$ we define 
\begin{align*}
  \nu_F = \sum_{z\in\VV_F}\eta_z - (n-1)d_F.
\end{align*}
For space $P^0(\FF_T)$ we use the characteristic functions $\chi_F|_{F'} = \delta_{F,F'}$ for $F'\in \FF_T$ as basis functions. 

\begin{lemma}\label{lem:dualbasis}
  We have 
  \begin{align}\label{eq:basisdual}
    \dual{\nu_F}{\chi_{F'}}_{\partial T} = |F| \delta_{F,F'} \quad\forall F,F'\in\FF_T
  \end{align}
  and
  \begin{align*}
    \linhull\set{\nu_F}{F\in\FF_T} = P^1(T), \quad \linhull\set{\nu_F|_{\partial T}}{F\in\FF_T} = P_c^1(\FF_T).
  \end{align*}
\end{lemma}
\begin{proof}
  Note that~\eqref{eq:basisdual} implies that $\nu_F|_{\partial T}$, $F\in\FF_T$ are linearly independent. This implies the last two assertions because $\dim (P^1(T)) = n+1 = \dim (P_c^1(\FF_T))$. It only remains to prove~\eqref{eq:basisdual}: Let $F\in\FF_T$. From its definition we see that $\nu_F|_F = 1$, thus, $\dual{\nu_F}{\chi_{F}}_{\partial T} = |F|$. Let $F'\in \FF_T\setminus\{F\}$. 
  Using $\int_{F'} \eta_z \,\di x = |F'|n^{-1}$ for $z\in \VV_{F'}$ one verifies that
  \begin{align*}
    \dual{\nu_F}{\chi_{F'}}_{\partial T} &= \sum_{z\in\VV_F\cap\VV_{F'}}\int_{F'} \eta_z\,\di x - (n-1)\int_{F'}d_F \,\di x 
    \\&= (n-1)|F'|n^{-1} - (n-1)|F'|n^{-1} = 0,
  \end{align*}
  finishing the proof. 
\end{proof}
Let $\RT^0(T) =\set{\ppsi\in \LL^2(T)}{\ppsi = \aalpha + \beta \xx, \, \aalpha\in\R^n, \beta\in \R}$ denote the lowest-order Raviart--Thomas space where $\xx\colon T\to\R^n$, $z\mapsto z$.
Let $\ppsi_{F}\in \RT^0(T)$ denote the canonical Raviart--Thomas basis function with
\begin{align*}
  \ppsi_F\cdot\normal_T|_{F'} = \delta_{F,F'} \quad\forall F,F'\in \FF_T
\end{align*}
and $\norm{\ppsi_F}{T} \eqsim |T|^{1/2}$. One verifies the explicit representation $\ppsi_F(z) = \frac{|F|}{n|T|}(z-z_F)$.
Note that by Lemma~\ref{lem:dualbasis} we have that
\begin{align*}
  \dual{\ppsi_F\cdot\normal_T}{\nu_{F'}}_{\partial T} = |F|\delta_{F,F'} \quad\forall F,F'\in\FF_T.
\end{align*}
We also use Bernardi--Raugel elements, see~\cite{BernardiRaugel85},
\begin{align*}
  \eeta_F := \eta_F\normal_F, \quad F\in\FF_T
\end{align*}
for which we get
\begin{align*}
  \dual{\eeta_F\cdot\normal_T}{\nu_{F'}}_{\partial T} = \dual{\eeta_F\cdot\normal_T}{1}_F \delta_{F,F'} 
  \eqsim |F| \delta_{F,F'} \quad\forall F,F'\in\FF_T.
\end{align*}
We also define edge based functions: Fix a vertex $z_*\in\VV_T$ and set $\EE_* = \EE_{z_*}$. 
Clearly, $\tangential_E$ ($E\in\EE_*$) are linearly independent and span $\R^n$. 
Let $\ssigma_E$, ($E\in\EE_*$) denote a basis of $\PP^0(T)$ such that (for any $z\in T$)
\begin{align*}
  \ssigma_E(z)\cdot\tangential_{E'} = \delta_{E,E'} \quad\forall E,E'\in\EE_*.
\end{align*}
We have that $|\ssigma_E(z)|\eqsim 1$ with constants only depending on the shape of $T$. 
With these preparations we define edge functions by $\eta_E = \prod_{z\in\VV_T\cap\overline{E}}\eta_{z}$ and tangential edge functions by
\begin{align*}
  \eeta_E := \eta_E\tangential_E \quad\forall E\in \EE_*.
\end{align*}
These functions play the role of element bubble functions in $\Hdivset{T}$ as can be seen from the next result.
\begin{lemma}\label{lem:rotatedBR}
  We have $\eeta_E\in \PP^{2}(T)$ and 
  \begin{align*}
    \eeta_E\cdot\normal_T|_{\partial T} = 0, \quad
    \ip{\ssigma_E}{\eeta_{E'}}_T = \ip{\ssigma_E}{\eeta_{E}}_T\delta_{E,E'}\eqsim |T|\delta_{E,E'}\quad\forall E,E'\in\EE_*.
  \end{align*}
\end{lemma}
\begin{proof}
  Clearly, $\eta_E\in P^2(T)$, thus, $\eeta_E\in \PP^2(T)$. Note that $\eta_E|_{\partial T}$ is supported on $n-1$ faces, say $F_j$, $j=1,\dots,n-1$. For these faces we also have $\tangential_E \subseteq F_j$ yielding $\tangential_E\cdot\normal_{F_j}=0$. We conclude $\eeta_E\cdot\normal_T|_{\partial T} = 0$.
  The final assertion follows from the definition of $\ssigma_E$ and $\eeta_E$ and scaling arguments.
\end{proof}

\subsection{Higher-order basis functions}\label{sec:ho}
For $F\in\FF_T$ let $\widetilde\chi_{F,j}\in P^p(T)$, $j=1,\dots,\dim(P^p(F))$ be such that $\widetilde\chi_{F,j}|_F$ is a basis of $P^p(F)$ with $\norm{\widetilde\chi_{F,j}}\infty = \norm{\widetilde\chi_{F,j}|_F}\infty \eqsim 1$ and define
\begin{align*}
  \eta_{F,j} &= \eta_F\widetilde\chi_{F,j}, \quad j=1,\ldots,\dim P^p(F), \, F\in\FF_T.
\end{align*}
Let $\chi_{F,j}\in P^p(\FF_T)$, $j=1,\dots,\dim(P^p(F))$, be such that $\chi_{F,j}|_{F'}=0$ for $F'\in\FF_T\setminus\{F\}$ and
\begin{align*}
  \dual{\chi_{F,j}}{\eta_{F,k}}_F = \delta_{j,k}|F|, \quad j,k=1,\ldots,\dim(P^p(F)).
\end{align*}
Let $\widetilde\chi_{T,j}\in P^p(T)$, $j=1,\dots,\dim(P^p(T))$, denote a basis of $P^p(T)$ with $\norm{\widetilde\chi_{T,j}}\infty\eqsim 1$ and define
\begin{align*}
  \eta_{T,j} = \eta_T\widetilde\chi_{T,j}, \quad j=1,\dots,\dim(P^p(T)).
\end{align*}
Furthermore, let $\chi_{T,j}$, $j=1,\dots,\dim(P^p(T))$ be such that
\begin{align*}
  \ip{\chi_{T,j}}{\eta_{T,k}}_T = |T|\delta_{j,k}, \quad j,k = 1,\dots,\dim(P^p(T)).
\end{align*}
By scaling arguments one verifies that $\norm{\chi_{F,j}}\infty\eqsim 1$ and $\norm{\chi_{T,j}}\infty \eqsim 1$. 

Let $\widetilde P^{p}(T)$ denote the orthogonal complement of $P_b^{p}(T) = \set{v\in P^{p}(T)}{v|_{\partial T}=0}$ in $P^{p}(T)$. 
Let $\widetilde\nu_{\partial T,j}$, $j=1,\dots,\dim(\widetilde P^{p+1}(T))$, denote a basis of $\widetilde P^{p+1}(T)$ with $\norm{\widetilde \nu_{\partial T,j}}\infty \eqsim 1$.
Furthermore, let $\nu_{\partial T,j}\in \widetilde P^{p+1}(T)$, $j=1,\dots,\dim(\widetilde P^{p+1}(T))$, denote a basis with 
\begin{align*}
  \dual{\nu_{\partial T,j}}{\widetilde\nu_{\partial T,k}}_{\partial T} = |\partial T| \delta_{j,k} \quad
  j,k = 1,\dots,\dim(\widetilde P^{p+1}(T)).
\end{align*}
One verifies that $\norm{\nu_{\partial T,j}}\infty \eqsim 1$ by scaling arguments. 
Define $\ppsi_{\partial T} = \sum_{F\in\FF_T}\ppsi_F$ and note that $\ppsi_{\partial T}\cdot\normal_T|_{\partial T} = 1$. 
Define 
\begin{align*}
  \ppsi_{\partial T,j} = \ppsi_{\partial T} \widetilde\nu_{\partial T,j} \quad j=1,\dots,\dim(\widetilde P^{p+1}(T)).
\end{align*}
Thus, by our previous considerations, $\dual{\ppsi_{\partial T,j}\cdot\normal_T}{\nu_{\partial T,k}}_{\partial T} = |\partial T| \delta_{j,k}$, $j,k=1,\dots,\dim(\widetilde P^{p+1}(T))$.

We also define higher order edge functions: For $E\in\EE_*$, $j=1,\dots,\dim(P^{p}(T))$, define
\begin{align*}
  \eeta_{E,j} &= \eeta_E\widetilde\chi_{T,j}.
\end{align*}
For $E\in\EE_*$ let $\chi_{E,j}\in P^p(T)$, $j=1,\dots,\dim(P^p(T))$, denote a basis with
\begin{align*}
  \ip{\chi_{E,j}}{\widetilde\chi_{T,k}\eta_{E}}_T = |T|\delta_{j,k} \quad j,k=1,\dots,\dim(P^p(T)).
\end{align*}
One verifies that $\norm{\chi_{E,j}}\infty \eqsim 1$.
Furthermore, define
\begin{align*}
  \ssigma_{E,j} = \ssigma_E \chi_{E,j} \quad j=1,\dots,\dim(P^p(T)), \, E\in\EE_*.
\end{align*}
The proof of the next result follows the arguments given in Lemma~\ref{lem:rotatedBR} together with the aforegoing definitions and is thus omitted.
\begin{lemma}
  We have that $\eeta_{E,j}\cdot\normal_T|_{\partial T} = 0$, $\eeta_{E,j}\in\PP^{p+2}(T)$, and
  \begin{align*}
    \ip{\ssigma_{E,j}}{\eeta_{E',k}}_T =  \delta_{E,E'}\delta_{j,k} |T|
  \end{align*}
  for all $E,E'\in\EE_*$, $j,k=1,\dots,\dim(P^{p}(T))$.
  \qed
\end{lemma}

\subsection{Modified face bubble functions}\label{sec:modBubble}
We introduce modified face bubble functions. Before we come to their definition and analysis we state the following result:
\begin{lemma}\label{lem:phiKappa}
  Let $R>0$.
  Consider $R\geq\kappa>0$ and the function $\phi_\kappa\colon [0,1]\to [0,1]$, $t\mapsto e^{-t/\kappa}$. 
  Then, 
  \begin{align*}
    \norm{\phi_\kappa}{L^2(0,1)} \eqsim \kappa^{1/2}, \quad \norm{\phi_\kappa'}{L^2(0,1)} \eqsim \kappa^{-1/2}, 
    \quad \phi_\kappa(0) = 1.
  \end{align*}
  The hidden constants only depend on $R$.
\end{lemma}
\begin{proof}
  The results follow from straightforward calculations.
\end{proof}

Recall that $d_F=\eta_{z_F}$. We can interpret $d_F$ as a relative distance function that is $0$ when restricted to $F$ and $1$ when evaluated at the vertex opposite to $F$. Considering $\phi:= \phi_{\alpha/h_T}$, i.e., 
\begin{align*}
  \phi(t) = e^{-h_Tt/\alpha},
\end{align*}
define for $F\in\FF_T$ the modified face bubble function by
\begin{align}\label{eq:def:modbubble}
  \eta_{\alpha,F} := (\phi\circ d_F)\eta_F.
\end{align}
Some basic properties of this modified function are given in the next result. A visualization of $\eta_{\alpha,F}$ is presented in Figure~\ref{fig:modbubble}.
\begin{figure}
  \begin{center}\includegraphics[width=0.8\textwidth]{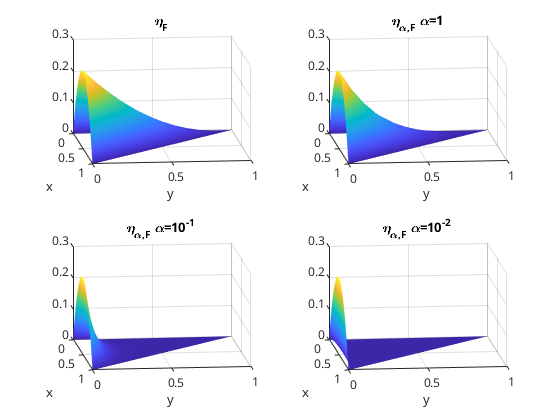}\end{center}
  \caption{Visualization of the face bubble functions $\eta_F$ and $\eta_{\alpha,F}$ on the reference element $\Tref$ and face $F=(0,1)\times\{0\}$.}
  \label{fig:modbubble}
\end{figure}

\begin{lemma}\label{lem:modbubble}
  Suppose that $0<\alpha\lesssim h_T$.
  For any $F\in\FF_T$ we have that
  \begin{align*}
    \norm{\eta_{\alpha,F}}T \lesssim |T|^{1/2} \left(\frac\alpha{h_T}\right)^{1/2}, \quad
    \norm{\nabla \eta_{\alpha,F}}T \lesssim h_T^{-1}|T|^{1/2} \left(\frac\alpha{h_T}\right)^{-1/2}, \quad
    \eta_{\alpha,F}|_{\partial T} = \eta_F|_{\partial T}.
  \end{align*}
\end{lemma}
\begin{proof}
  The identity $\eta_{\alpha,F}|_{\partial T} = \eta_F|_{\partial T}$ follows since $\phi(0) = 1$ and $\eta_F|_{\partial T\setminus F'} = 0$ for $F'\in\FF_T\setminus\{F\}$.

  We show the details for $n=2$. For $n\geq3$ we may argue similarly. 
  Let $A_T\colon \Tref\to T$ denote the affine element mapping.
  W.l.o.g. let $F\in\FF_T$ be the face such $A_T\colon \Tref\to T$ maps $F$ to the edge $(0,1)\times\{0\}$. 
  Then, $d_F\circ A_T(\widehat x,\widehat y) = \widehat y$ for $(\widehat x,\widehat y)\in \Tref$.
  Moreover, 
  \begin{align*}
    \eta_{\alpha,F} \circ A_T(\widehat x,\widehat y) &= \widehat\eta_\alpha (\widehat x,\widehat y) := e^{-\widehat y h_T/\alpha} (1-\widehat x-\widehat y)\widehat x.
  \end{align*}
  The remaining assertions follow by standard calculations, e.g., 
  \begin{align*}
    \norm{\eta_{\alpha,F}}T^2 = 2|T| \norm{\widehat\eta_\alpha}{\Tref}^2 \lesssim \frac\alpha{h_T} |T|.
  \end{align*}
\end{proof}

With the same logic as for the definition of $\eta_{F,j}$ we define 
\begin{align*}
  \eta_{\alpha,F,j} = \eta_{\alpha,F}\widetilde\chi_{F,j}, \quad j=1,\dots,\dim(P^p(F)), \, F\in\FF_T.
\end{align*}
The same proof as for Lemma~\ref{lem:modbubble} shows
\begin{lemma}\label{lem:modbubble:ho}
  Suppose that $\alpha\lesssim h_T$. Then, for any $F,j$ the function $\eta_{\alpha,F,j}$ satisfies the assertions of Lemma~\ref{lem:modbubble} (replacing $\eta_{\alpha,F}$ by $\eta_{\alpha,F,j}$ and $\eta_F$ by $\eta_{F,j}$).
\end{lemma}

Define the modified Bernardi--Raugel elements by
\begin{align*}
  \eeta_{\alpha,F} := \eta_{\alpha,F}\normal_F = (\phi\circ d_F)\eeta_F.
\end{align*}
Some important properties of $\eeta_{\alpha,F}$ follow directly from its definition and are summarized in the next result.
Its proof follows the same ideas as the proof of Lemma~\ref{lem:modbubble}.
\begin{lemma}\label{lem:modRB}
  For $F\in\FF_T$ we have $\eeta_{\alpha,F}\cdot\normal_T|_{\partial T} = \eeta_F\cdot\normal_T|_{\partial T}$
  and if $\alpha\lesssim h_T$, then,
  \begin{align*}
    \norm{\eeta_{\alpha,F}}T \lesssim |T|^{1/2} \left(\frac{\alpha}{h_T}\right)^{1/2}, 
    \quad \norm{\div\eeta_{\alpha,F}}{T}\lesssim |T|^{1/2}h_T^{-1} \left(\frac{\alpha}{h_T}\right)^{-1/2}.
  \end{align*}
\end{lemma}
We also need higher order variants: Set $\eeta_{\alpha,\partial T} = \sum_{F\in\FF_T}\eeta_{\alpha,F}$ and 
\begin{align*}
  \eeta_{\alpha,\partial T,j} = \eeta_{\alpha,\partial T}\widetilde\nu_{\partial T,j} \quad j=1,\dots,\dim(P_c^{p+1}(\FF_T)).
\end{align*}
Let $\chi_{\partial T,j}\in P_c^{p+1}(\FF_T)$, $j=1,\dots,\dim(P_c^{p+1}(\FF_T))$ denote a basis of $P_c^{p+1}(\FF_T)$ with
\begin{align*}
  \dual{\eeta_{\alpha,\partial T,j}\cdot\normal_T}{\chi_{\partial T,k}}_{\partial T} = |\partial T|\delta_{j,k}, \quad
  j,k = 1,\dots,\dim(P_c^{p+1}(\partial T)).
\end{align*}
The proof of the next result follows the proof of Lemma~\ref{lem:modRB}.
\begin{lemma}\label{lem:modRB:ho}
  The boundedness estimates of Lemma~\ref{lem:modRB} hold with $\eeta_{\alpha,F}$ replaced by $\eeta_{\alpha,\partial T,j}$. 
  \qed
\end{lemma}

To close this section and to have a better overview we summarize the most important basis functions used in the remainder of this work in Table~\ref{tab:basis}.
\begin{table}[htbp]
\centering
\begin{tabular}{|c|c|c|c|c|c|}\hline
 & function & space & basis & definition & property  \\ \hline\hline
 \multirow{ 6}{*}{\rotatebox[origin=c]{90}{low order}} & $\chi_F$ & $P^0(\FF_T)$ & yes & $\chi_F|_{F'}=\delta_{F,F'}$ & -- \\ \cline{2-6} 
 & $\nu_F$ & $P^1(T)$ & yes & $\sum_{z\in\VV_F}\eta_z-(n-1)d_F$ & $\dual{\nu_F}{\chi_{F'}}_{\partial T} = |F|\delta_{F,F'}$ \\ \cline{2-6}
 & $\ppsi_F$ & $\RT^0(T)$ & yes & $\ppsi_F\cdot\normal_T|_{\partial T} = \chi_F$ & $\dual{\ppsi_F\cdot\normal_T}{\nu_{F'}}_{\partial T} = |F|\delta_{F,F'}$ \\ \cline{2-6}
 & $\eeta_F$ & $\PP^p(T)$ & no & $\eta_F\normal_F$ & $\dual{\eeta_F\cdot\normal_T}{\nu_{F'}}_{\partial T} = \dual{\eta_F}1_F\delta_{F,F'}$ \\ \cline{2-6}
& $\ssigma_E$ & $\PP^0(T)$ & yes &-- & $\ssigma_E\cdot\tangential_{E'} = \delta_{E,E'}$  \\ \cline{2-6}
& $\eeta_E$ & $\PP^2(T)$ & no & $\eta_E\tangential_E$ & $\ip{\ssigma_E}{\eeta_{E'}}_T = \ip{\ssigma_E}{\eeta_{E}}_T\delta_{E,E'}$  \\
\hline\hline
\multirow{ 10}{*}{\rotatebox[origin=c]{90}{higher order}} & $\chi_{T,j}$, $\widetilde\chi_{T,j}$ & $P^p(T)$ & yes & -- & -- \\ \cline{2-6} 
& $\eta_{T,j}$ & $P^{p+n+1}(T)$ & no & $\eta_{T}\widetilde\chi_{T,j}$ & $\ip{\eta_{T,j}}{\chi_{T,k}}_T=|T|\delta_{j,k}$ \\ \cline{2-6}
& $\widetilde\chi_{F,j}$ & $P^p(T)$ & yes & -- & -- \\ \cline{2-6}
& $\chi_{F,j}$ & $P^p(\FF_T)$ & yes & -- & -- \\ \cline{2-6}
& $\eta_{F,j}$ & $P^{p+n}(T)$ & no & $\eta_F\widetilde\chi_{F,j}$ & $\dual{\eta_{F,j}}{\chi_{F',k}}_{\partial T} = \delta_{j,k}\delta_{F,F'}|F|$  \\ \cline{2-6}
& $\nu_{\partial T,j}$, $\widetilde\nu_{\partial T,j}$ & $\widetilde P^{p+1}(T)$ & yes &--  & $\dual{\nu_{\partial T,j}}{\widetilde\nu_{\partial T,k}}_{\partial T} = |\partial T| \delta_{j,k}$ \\  \cline{2-6}
& $\ppsi_{\partial T,j}$ & $\PP^{p+2}(T)$ & no & $\widetilde\nu_{\partial T,j}\sum_{F\in\FF_T}\ppsi_F$ & $\dual{\ppsi_{\partial T,j}\cdot\normal_T}{\nu_{\partial T,k}}_{\partial T} = \delta_{j,k}|\partial T|$  \\  \cline{2-6}
& $\chi_{E,j}$ & $P^p(T)$ & yes & -- & $\ip{\chi_{E,j}}{\widetilde\chi_{T,k}\eta_E}_{T} = \delta_{j,k}|T|$  \\ \cline{2-6} 
& $\ssigma_{E,j}$ & $\PP^p(T)$ & yes & $\ssigma_E\chi_{E,j}$ & --  \\ \cline{2-6}
& $\eeta_{E,j}$ & $\PP^{p+2}(T)$ & no & $\eeta_E\widetilde\chi_{T,j}$ & $\ip{\ssigma_{E,j}}{\eeta_{E',k}}_{T} = \delta_{j,k}\delta_{E,E'}|T|$  \\ 
\hline\hline
\multirow{ 6}{*}{\rotatebox[origin=c]{90}{modified}} & $\phi$ & -- & -- & $t\mapsto e^{-h_T/\alpha t}$ & -- \\ \cline{2-6} 
& $\eta_{\alpha,F}$ & -- & -- & $(\phi\circ d_F)\eta_F$ & $\eta_{\alpha,F}|_{\partial T} = \eta_F|_{\partial T}$ \\ \cline{2-6}
& $\eta_{\alpha,F,j}$ & -- & -- & $\eta_{\alpha,F}\widetilde\chi_{F,j}$ & $\eta_{\alpha,F,j}|_{\partial T} = \eta_{F,j}|_{\partial T}$ \\ \cline{2-6}
& $\eeta_{\alpha,F}$ & -- & -- & $\eta_{\alpha,F}\normal_F$ & $\eeta_{\alpha,F}\cdot\normal_T|_{\partial T} = \eeta_{F}\cdot\normal_T|_{\partial T}$ \\ \cline{2-6}
& $\chi_{\partial T,j}$ & $P^{p+1}_c(\FF_T)$ & yes & -- & --  \\ \cline{2-6}
& $\eeta_{\alpha,\partial T,j}$ & -- & -- & $\eeta_{\alpha,F}\widetilde\nu_{\partial T,j}$ & $\dual{\eeta_{\alpha,F,j}\cdot\normal_T}{\chi_{\partial T,k}}_{\partial T} = |\partial T|\delta_{j,k}$  \\
\hline
\end{tabular}
\caption{Overview of special functions together with some of their main properties. 
  Here, $\eta_z$ is the canonical Lagrange basis function of $P^1(T)$, $\eta_E$, $\eta_F$, $\eta_T$ are edge, face, and element bubble functions, respectively, and $d_F = \eta_{z_F}$. In column \textit{basis} we indicate whether the respective family of functions generates the indicated space. 
For a detailed description we refer to Sections~\ref{sec:low}--\ref{sec:modBubble}.}
\label{tab:basis}
\end{table}

\section{Fortin operator in $H^1(T)$}\label{sec:fortin}
\noindent
We consider a fixed parameter $\alpha>0$ and space $H^1(T)$ equipped with the (squared) norm
\begin{align*}
  \norm{v}{T,\alpha}^2 := \norm{v}{T}^2 + \alpha^2 \norm{\nabla v}T^2.
\end{align*}
The idea of this section is to construct Fortin operators, say $\PiF^\nabla\colon H^1(T)\to V_h^\nabla$ (with $V_h^\nabla\subset H^1(T)$ being some finite-dimensional subspace) such that, for a fixed $p\in\N_0$ and for all $v\in H^1(T)$,
\begin{subequations}\label{eq:fortinprop}
\begin{align}\label{eq:fortinprop:bound}
  \norm{\PiF^\nabla v}{T,\alpha} &\leq C_\mathrm{F} \norm{v}{T,\alpha}, \\
  \label{eq:fortinprop:a}
  \dual{\sigma}{v-\PiF^\nabla v}_{\partial T} &= 0 \quad\forall \sigma\in P^p(\FF_T), \\
  \label{eq:fortinprop:b}
  \ip{u}{v-\PiF^\nabla v}_T &= 0 \quad\forall u\in P^p(T) 
\end{align}
with $C_\mathrm{F}>0$ independent of $\alpha$, $h_T$ (but possibly dependent on $p$).
Note that~\eqref{eq:fortinprop:a}--\eqref{eq:fortinprop:b} imply
\begin{align}\label{eq:fortinprop:c}
  \ip{\ssigma}{\nabla(1-\PiF^\nabla)v}_T = 0 \quad\forall \ssigma\in\PP^{p}(T).
\end{align}
\end{subequations}
This can be seen from integration by parts: Take $\ssigma\in \PP^p(T)$, $v\in H^1(T)$. Then, $\div\ssigma\in P^{p-1}(T)$ and
\begin{align*}
  \ip{\ssigma}{\nabla \PiF^\nabla v}_T &= -\ip{\div\ssigma}{\PiF^\nabla v}_T + \dual{\ssigma\cdot\normal_T}{\PiF^\nabla v}_{\partial T}
  \\
  &= -\ip{\div\ssigma}{v}_T + \dual{\ssigma\cdot\normal_T}{v}_{\partial T} = \ip{\ssigma}{\nabla v}_T.
\end{align*}
From the last identities we also see that the weaker condition
\begin{align}\tag{\ref{eq:fortinprop:b}'}\label{eq:fortinprop:b:alt}
  \ip{u}{v-\PiF^\nabla v}_T &= 0 \quad\forall u\in P^{p-1}(T)
\end{align}
would be sufficient to conclude~\eqref{eq:fortinprop:c}. 
However, depending on the problem, condition~\eqref{eq:fortinprop:b} is needed, e.g., in the presence of reaction terms as in the DPG method in Section~\ref{sec:num}.
We stress that our Fortin operator can be easily modified to satisfy~\eqref{eq:fortinprop:b:alt} only.

\subsection{Constructions for moderate parameter}\label{sec:grad:moderate}
Define the space
\begin{align*}
  \widetilde V_{hp}^\nabla = P^0(T) + \linhull\set{\eta_{F,j}}{j=1,\dots,\dim(P^p(F)), F\in\FF_T}
\end{align*}
and operator $\PiFptilde^\nabla \colon H^1(T)\to \widetilde V_{hp}^\nabla$ for $v\in H^1(T)$ by
\begin{align*}
  \PiFptilde^\nabla v = \Pi_T^0 v + \sum_{F\in\FF_T} \sum_{j=1}^{\dim(P^p(F))} \frac{\dual{\chi_{F,j}}{(1-\Pi_T^0)v}_F}{\dual{\chi_{F,j}}{\eta_{F,j}}_F} \eta_{F,j}.
\end{align*}
  
The following result collects its main properties.
\begin{lemma}\label{lem:fortin:normal}
  Operator $\PiF^\nabla=\PiFptilde^\nabla$ is idempotent on $P^0(T)$, satisfies property~\eqref{eq:fortinprop:a} and
  \begin{align*}
    \norm{\PiFptilde^\nabla v}T &\lesssim \norm{v}T + h_T\norm{\nabla v}T, \quad
    \norm{\nabla \PiFptilde^\nabla v}T \lesssim \norm{\nabla v}T, 
    \quad\norm{(1-\PiFptilde^\nabla)v}T \lesssim h_T\norm{\nabla v}T
  \end{align*}
  for all $v\in H^1(T)$.
\end{lemma}
\begin{proof}
  Idempotency can be seen from the definition, since $v\in P^0(T)$ implies that $\Pi_T^0 v =v$, thus, $(1-\Pi_T^0)v = 0$.

  Let $v\in H^1(T)$. To see~\eqref{eq:fortinprop:a} we employ the orthogonality property $\dual{\eta_{F,j}}{\chi_{F',k}}_{\partial T} = \delta_{F,F'}\delta_{j,k}|F|$ to get 
  \begin{align*}
    \dual{\chi_{F',k}}{v-\PiFtilde^\nabla v}_{\partial T} &= \dual{\chi_{F',k}}{v-\Pi_T^0v}_{\partial T} - \sum_{F\in\FF_T} \sum_{j=1}^{\dim(P^p(F))} 
    \frac{\dual{\chi_{F,j}}{(1-\Pi_T^0)v}_F}{\dual{\chi_{F,j}}{\eta_{F,j}}_F}\dual{\chi_{F',k}}{\eta_{F,j}}_{\partial T}
    \\
    &= \dual{\chi_{F',k}}{v-\Pi_T^0v}_{\partial T} - \dual{\chi_{F',k}}{(1-\Pi_T^0)v}_{\partial T} = 0
  \end{align*}
  Since $F'$ and $k$ were arbitrary, condition~\eqref{eq:fortinprop:a} follows. 

  The boundedness follows from the triangle inequality, the Cauchy--Schwarz inequality and boundedness of $\Pi_T^0$, i.e.,
  \begin{align*}
    \norm{\PiFptilde^\nabla v}{T} \leq \norm{v}T + \sum_{F\in\FF_T} \sum_{j=1}^{\dim(P^p(T))}\frac{\norm{\chi_{F,j}}F\norm{v-\Pi_T^0 v}F}{\dual{\chi_{F,j}}{\eta_{F,j}}_F} \norm{\eta_{F,j}}T.
  \end{align*}
  Note that $\dual{\chi_{F,j}}{\eta_{F,j}}_F = |F|$, $\norm{\eta_{F,j}}T \eqsim |T|^{1/2}$, $\norm{\chi_{F,j}}F\eqsim |F|^{1/2}$ which follows by standard scaling arguments and the properties of the basis functions discussed in Section~\ref{sec:basis}. 
  Applying the trace inequality~\eqref{eq:traceineq} we see that
  \begin{align*}
    \frac{\norm{\chi_{F,j}}F\norm{v-\Pi_T^0 v}F}{\dual{\chi_{F,j}}{\eta_{F,j}}_F} \norm{\eta_{F,j}}T
    &\eqsim |F|^{-1/2}|T|^{1/2} \norm{v-\Pi_T^0 v}{F} \lesssim h_T^{1/2} h_T^{1/2} \norm{\nabla v}T.
  \end{align*}
  Thus, we conclude that $\norm{\PiFptilde^\nabla v}T \lesssim \norm{v}T + h_T\norm{\nabla v}T$. 
  Then, with similar arguments but using the inverse estimate $\norm{\nabla \eta_{F,j}}T \lesssim h_T^{-1}|T|^{1/2}$, we see that ($\nabla \Pi_T^0 v=0$)
  \begin{align*}
    \norm{\nabla \PiFptilde^\nabla v}T \leq \sum_{F\in\FF_T} \sum_{j=1}^{\dim(P^p(F))} \frac{\norm{\chi_{F,j}}F\norm{v-\Pi_T^0 v}F}{\dual{\chi_{F,j}}{\eta_{F,j}}_F} \norm{\nabla \eta_{F,j}}T \lesssim \norm{\nabla v}T.
  \end{align*}
  Finally, the approximation property is derived by using the idempotency, and the established boundedness estimates, i.e., 
  \begin{align*}
    \norm{(1-\PiFptilde^\nabla)v}T = \norm{(1-\PiFptilde^\nabla)(v-\Pi_T^0v)}T \lesssim \norm{v-\Pi_T^0v}T + h_T \norm{\nabla(v-\Pi_T^0v)}T 
    \lesssim h_T\norm{\nabla v}T.
  \end{align*}
  This concludes the proof.
\end{proof} 

To obtain an operator that also satisfies property~\eqref{eq:fortinprop:b} we consider slight modifications by adding a correction term based on element bubbles.
Define the space
\begin{align*}
  V_{hp}^{\nabla} = \widetilde V_{hp}^{\nabla} + \linhull\set{\eta_{T,j}}{j=1,\dots,\dim(P^p(T))}
\end{align*}
and the operator $\PiFp^\nabla\colon H^1(T)\to V_{hp}^\nabla$ for all $v\in H^1(T)$ by
\begin{align*}
  \PiFp^\nabla v = \PiFptilde^\nabla v + \sum_{j=1}^{\dim(P^p(T))} \frac{\ip{\chi_{T,j}}{(1-\PiFptilde^\nabla)v}_T}{\ip{\chi_{T,j}}{\eta_{T,j}}_T}\eta_{T,j}.
\end{align*}

\begin{theorem}\label{thm:fortin:normal}
  Operator $\PiF^\nabla = \PiFp^\nabla$ is idempotent on $P^0(T)$, satisfies~\eqref{eq:fortinprop:a}--\eqref{eq:fortinprop:b} and 
  \begin{align*}
    \norm{\PiFp^\nabla v}T &\lesssim \norm{v}T + h_T\norm{\nabla v}T, \quad
    \norm{\nabla \PiFp^\nabla v}T \lesssim \norm{\nabla v}T, 
    \quad \norm{(1-\PiFp^\nabla)v}T \lesssim h_T\norm{\nabla v}T
  \end{align*}
  for all $v\in H^1(T)$.
\end{theorem}
\begin{proof}
  The idempotency on $P^0(T)$ follows from the idempotency of $\PiFptilde^{\nabla}$ (Lemma~\ref{lem:fortin:normal}).
  Statement~\eqref{eq:fortinprop:a} follows also from Lemma~\ref{lem:fortin:normal} since the element bubbles $\eta_{T,j}$ vanish on the boundary and, therefore, $\PiFp^\nabla v|_{\partial T} = \PiFptilde^{\nabla}v|_{\partial T}$. 
  To see~\eqref{eq:fortinprop:b} a simple calculation using the orthogonality $\ip{\chi_{T,j}}{\eta_{T,k}}_T = |T|\delta_{j,k}$ yields
  \begin{align*}
    \ip{\chi_{T,k}}{(1-\PiF^\nabla)v}_T &= \ip{\chi_{T,k}}{(1-\PiFptilde^{\nabla})v}_T - \sum_{j=1}^{\dim(P^p(T))}\frac{\ip{\chi_{T,k}}{(1-\PiFptilde^{\nabla})v}_T}{\ip{\chi_{T,j}}{\eta_{T,j}}_T}\ip{\chi_{T,k}}{\eta_{T,j}}_T \\
    &= \ip{\chi_{T,k}}{(1-\PiFptilde^{\nabla})v}_T -\ip{\chi_{T,k}}{(1-\PiFptilde^{\nabla})v}_T = 0.
  \end{align*}
  It remains to prove the boundedness estimates which follow --- besides standard arguments --- from the boundedness estimates of $\PiFtilde^{\nabla,j}$ (see Lemma~\ref{lem:fortin:normal}). 
  First, using the triangle inequality, Cauchy--Schwarz inequality and scaling arguments we estimate
  \begin{align*}
    \norm{\PiFp^\nabla v}T &\lesssim \norm{\PiFptilde^{\nabla}v}T + \sum_{j=1}^{\dim(P^p(T))}|T|^{-1/2}\norm{v-\PiFptilde^{\nabla}v}T \norm{\eta_{T,j}}T
    \\
    &\lesssim \norm{\PiFptilde^{\nabla}v}T + \norm{v-\PiFptilde^{\nabla}v}T \lesssim \norm{v}T + h_T\norm{\nabla v}T.
  \end{align*}
  The gradient contribution is estimated by employing the inverse estimate $\norm{\nabla \eta_{T,j}}{T}\lesssim h_T^{-1}\norm{\eta_{T,j}}T$ together with Lemma~\ref{lem:fortin:normal} to give
  \begin{align*}
    \norm{\nabla\PiFp^\nabla v}{T} \lesssim \norm{\nabla\PiFptilde^{\nabla} v}T + h_T^{-1}\norm{v-\PiFptilde^{\nabla}v}T 
    \lesssim \lesssim \norm{\nabla v}T.
  \end{align*}
  The final assertion $\norm{(1-\PiFp^\nabla)v}T \lesssim h_T\norm{\nabla v}T$ follows as in Lemma~\ref{lem:fortin:normal}.
\end{proof}

\begin{corollary}
  Suppose that $h_T\lesssim \alpha$. 
  From Lemma~\ref{lem:fortin:normal} and Theorem~\ref{thm:fortin:normal} it follows that $\PiF^\nabla = \PiFptilde^\nabla$ and $\PiF^\nabla = \PiFp^\nabla$ satisfy~\eqref{eq:fortinprop:bound}.
  Particularly, $\PiF^\nabla = \PiFp^\nabla$ has properties~\eqref{eq:fortinprop}.
\end{corollary}

\begin{remark}\label{rem:supconv:grad}
  In general, discrete test spaces are chosen such that a Fortin operator exists, which not necessarily implies approximation results of the form
  \begin{align*}
    \min_{v_h\in V_h} \norm{v-v_h}{T} + \norm{\nabla(v-v_h)}T \lesssim h_T |v|_{2,T} \quad\text{for } v\in H^2(T). 
  \end{align*}
  Here, $|\cdot|_{2,T}$ denotes the $H^2(T)$ seminorm.
  For certain supercloseness results in the DPG method the latter approximation property is needed, see~\cite{SupConv2}.
  To ensure this property one can simply require $P^1(T)\subset V_h$.
\end{remark}

\subsection{Constructions for small parameter}
In this section we focus on the case $\alpha\lesssim h_T$. 
Let $\PiF^\nabla=\PiFptilde^\nabla$ or $\PiF^\nabla=\PiFp^\nabla$. By Lemma~\ref{lem:fortin:normal} resp. Theorem~\ref{thm:fortin:normal} we have the boundedness
\begin{align}\label{eq:bound:fortin}
  \norm{\PiF^\nabla v}T \lesssim \norm{v}T + h_T \norm{\nabla v}T \lesssim \max\{1,h_T\alpha^{-1}\} \left(\norm{v}T + \alpha\norm{\nabla v}T\right).
\end{align}
We conclude that $\norm{\PiF^\nabla v}{T,\alpha} \lesssim \max\{1,h_T\alpha^{-1}\} \norm{v}{T,\alpha}$.
This tells us that $\PiF^\nabla$ is only conditionally uniformly bounded. Particularly, for small parameters $\alpha$ and coarse meshes huge boundedness constants are expected so that robustness of the numerical methods is likely to be lost. This can actually be observed in our numerical experiments presented in Section~\ref{sec:num}.
Let us remark that the operators constructed in~\cite{practicalDPG,breakSpace,DemkowiczZanotti20} also satisfy~\eqref{eq:bound:fortin} and are not suited for small parameters, $\alpha\lesssim h_T$.

To overcome this problem we construct an operator based on the modified face bubble functions $\eta_{\alpha,F,j}$ instead of $\eta_{F,j}$.
The construction of the novel Fortin operators follows the definition of $\PiFptilde^\nabla$ and $\PiFp^\nabla$ replacing $\eta_{F,j}$ with the modified face bubble functions $\eta_{\alpha,F,j}$. First, set
\begin{align*}
  \widetilde V_{hp,\alpha}^\nabla &:= P^0(T) + \linhull\{\eta_{\alpha,F,j}\,:\, j=1,\dots,\dim(P^p(F)), F\in\FF_T\}, 
\end{align*}
and define $\PiFaptilde^\nabla \colon H^1(T)\to \widetilde V_{hp,\alpha}^\nabla$ for all $v\in H^1(T)$ by
\begin{align*}
  \PiFaptilde^\nabla v &:= \Pi_T^0 v + \sum_{F\in\FF_T} \sum_{j=1}^{\dim(P^p(F))} \frac{\dual{\chi_{F,j}}{(1-\Pi_T^0)v}_F}{\dual{\chi_{F,j}}{\eta_{\alpha,F,j}}_F} \eta_{\alpha,F,j}.
\end{align*}
Its main properties are given in
\begin{lemma}\label{lem:fortin}
  Operator $\PiF^\nabla = \PiFaptilde^\nabla$ satisfies~\eqref{eq:fortinprop:a} and is idempotent on $P^0(T)$.
  If $\alpha\lesssim h_T$, then
  \begin{align*}
    \norm{\PiFaptilde^\nabla v}{T,\alpha} \lesssim \norm{v}{T,\alpha}, \quad 
    \norm{v-\PiFaptilde^\nabla v}T &\lesssim h_T \norm{\nabla v}{T}
  \end{align*}
  for all $v\in H^1(T)$.
\end{lemma}
\begin{proof}
  The idempotency on $P^0(T)$ can be seen directly from the definition of the operator. 
  Noting that $\eta_{\alpha,F,j}|_{\partial T} = \eta_{F,j}|_{\partial T}$ for any $F,j$ the proof of Fortin property~\eqref{eq:fortinprop:a} follows as in Lemma~\ref{lem:fortin:normal}.
  
  It remains to prove the boundedness estimates. Suppose that $\alpha\lesssim h_T$. 
  First, using the triangle inequality and the Cauchy--Schwarz inequality together with the properties of the modified bubble (Lemma~\ref{lem:modbubble:ho}), $\norm{\chi_{F,j}}F \eqsim |F|^{1/2}$, 
  $\dual{\chi_{F,j}}{\eta_{\alpha,F,j}}_F = |F|$ we infer 
  \begin{align*}
    \norm{\PiFaptilde^\nabla v}T &\leq \norm{v}T + \sum_{F\in\FF_T} \sum_{j=1}^{\dim(P^p(F))} \frac{\norm{\chi_{F,j}}F\norm{v-\Pi_T^0v}F}{\dual{\chi_{F,j}}{\eta_{\alpha,F,j}}} \norm{\eta_{\alpha,F,j}}T
    \\  &\lesssim \norm{v}T + \sum_{F\in\FF_T} \norm{v-\Pi_T^0v}F |F|^{-1/2} |T|^{1/2} \alpha^{1/2}h_T^{-1/2}
    \eqsim \norm{v}T + \sum_{F\in\FF_T} \norm{v-\Pi_T^0v}F  \alpha^{1/2}.
  \end{align*}
  Then, with the multiplicative trace inequality~\eqref{eq:traceineq} and Young's inequality we further get 
  \begin{align*}
    \norm{v-\Pi_T^0v}F  \alpha^{1/2} &\lesssim \norm{v-\Pi_T^0 v}T^{1/2}\alpha^{1/2}\norm{\nabla v}T^{1/2}
    \lesssim \norm{v-\Pi_T^0 v}T + \alpha \norm{\nabla v}T. 
  \end{align*}
  Putting all the estimates together we infer that
  \begin{align*}
    \norm{\PiFaptilde^\nabla v}T \lesssim  \norm{v}T + \sum_{F\in\FF_T}\norm{v-\Pi_T^0v}F  \alpha^{1/2} \lesssim \norm{v}T + \alpha \norm{\nabla v}T.
  \end{align*}
  We are left with the gradient contribution of the norm. With similar arguments as before we get
  \begin{align*}
    \alpha\norm{\nabla \PiFaptilde^\nabla v}T &\leq \alpha\sum_{F\in\FF_T} \sum_{j=1}^{\dim(P^p(F))}  \frac{\norm{\chi_{F,j}}F\norm{v-\Pi_T^0v}F}{\dual{\chi_{F,j}}{\eta_{\alpha,F,j}}} \norm{\nabla \eta_{\alpha,F,j}}T
    \\
    &\lesssim \alpha\sum_{F\in\FF_T} \alpha^{-1/2} \norm{v-\Pi_T^0 v}F 
    \\
    &\lesssim \alpha^{1/2} \norm{v-\Pi_T^0v}T^{1/2}\norm{\nabla v}T^{1/2} \lesssim \norm{v}T + \alpha\norm{\nabla v}T.
  \end{align*}
  Combining all the estimates from above we conclude the boundedness of the Fortin operator. 

  Finally, to see the approximation property note that $\PiFaptilde^\nabla 1 = 1$, thus, 
  \begin{align*}
    \norm{(1-\PiFaptilde^\nabla)v}T &= \norm{(1-\PiFaptilde^\nabla)(v-\Pi_T^0v)}T \lesssim \norm{v-\Pi_T^0v}T + \alpha \norm{\nabla v}T \lesssim h_T \norm{\nabla v}T.
  \end{align*}
  Note that we have used that $\alpha\lesssim h_T$. 
\end{proof}

To ensure property~\eqref{eq:fortinprop:b} we follow the idea of the definition of $\PiFp^\nabla$ by adding correction terms based on element bubble functions. Contrary to the face bubble functions, the element bubble functions do not need to be modified. 
Set 
\begin{align*}
  V_{hp,\alpha}^\nabla = \widetilde V_{hp,\alpha}^\nabla+\linhull\set{\eta_{T,j}}{j=1,\dots,\dim(P^p(T))}
\end{align*}
and define $\PiFap^\nabla\colon H^1(T)\to V_{hp,\alpha}^\nabla$ for all $v\in H^1(T)$ by
\begin{align*}
  \PiFap^\nabla v:= \PiFaptilde^\nabla v + \sum_{j=1}^{\dim(P^p(T))} \frac{\ip{\chi_{T,j}}{(1-\PiFaptilde^\nabla)v}_T}{\ip{\chi_{T,j}}{\eta_{T,j}}_T}\eta_{T,j}.
\end{align*}
The following theorem is one of our main results.
\begin{theorem}\label{thm:fortin}
  Suppose that $\alpha\lesssim h_T$.
  Then, $\PiF^\nabla = \PiFap^\nabla$ is idempotent on $P^0(T)$ and satisfies~\eqref{eq:fortinprop}.
\end{theorem}
\begin{proof}
  Idempotency follows from the definition.
  Verifying~\eqref{eq:fortinprop:a}--\eqref{eq:fortinprop:b} follows as in Theorem~\ref{thm:fortin:normal}. 
  For the boundedness estimate we also use the arguments displayed in the proof of Theorem~\ref{thm:fortin:normal}. Note that with Lemma~\ref{lem:fortin} and $\norm{\nabla\eta_{T,j}}T\lesssim h_T^{-1}|T|^{1/2}$ we see that, for $v\in H^1(T)$,
  \begin{align*}
    \norm{\PiFap^\nabla v}{T,\alpha} &\lesssim \norm{\PiFaptilde^\nabla v}{T,\alpha} + 
    \sum_{j=1}^{\dim(P^p(T))}\frac{\ip{\chi_{T,j}}{v-\PiFaptilde^\nabla v}_T}{\ip{\chi_{T,j}}{\eta_{T,j}}_T} (\norm{\eta_{T,j}}T + \alpha\norm{\nabla \eta_{T,j}}{T})
    \\
    &\lesssim \norm{v}{T,\alpha} + |T|^{-1/2}\norm{v-\PiFaptilde^\nabla v}T |T|^{1/2} + |T|^{-1/2}\norm{v-\PiFaptilde^\nabla v}T \alpha h_T^{-1}|T|^{1/2}
    \\
    &\lesssim \norm{v}{T,\alpha} + \norm{v-\PiFaptilde^\nabla v}T.
  \end{align*}
  In the last step we used $\alpha\lesssim h_T$. Boundedness of $\PiFaptilde^\nabla$ finishes the proof.
\end{proof}

\subsection{Alternative operator for lowest-order spaces and moderate parameter}
In this section we construct a Fortin operator $\PiF^{\nabla}\colon H^1(T)\to V_{h}^{\nabla,0}$ such that~\eqref{eq:fortinprop} is satisfied for the lowest-order case $p=0$. In~\cite{CarstensenGHW14} it is shown that for a low-order DPG method for the Poisson problem, test space $V_h^\nabla = P^1(T)$ for the scalar test functions is sufficient to guarantee well-posedness. The authors of~\cite{CarstensenGHW14} did use different techniques. 
Here, we complement their results by the construction of a Fortin operator. Define the spaces
\begin{align*}
  \widetilde V_h^{\nabla,0} = P^1(T), \quad V_h^{\nabla,0} = \widetilde V_h^{\nabla,0} + \linhull\{\eta_T\}
\end{align*}
and operators $\PiFtilde^{\nabla,0}\colon H^1(T)\to \widetilde V_h^{\nabla,0}$, $\PiF^{\nabla,0}\colon H^1(T)\to V_h^{\nabla,0}$ for all $v\in H^1(T)$ by
\begin{align*}
  \PiFtilde^{\nabla,0} v &= \Pi_T^0v + \sum_{F\in\FF_T} \frac{\dual{1}{(1-\Pi_T^0)v}_F}{\dual{1}{\nu_F}_F}\nu_F, \\
  \PiF^{\nabla,0} v&= \PiFtilde^{\nabla,0} v + \frac{\ip{1}{(1-\PiFtilde^{\nabla,0})v}_T}{\ip{1}{\eta_T}_T} \eta_T.
\end{align*}
The analysis of these operators can be done as in Section~\ref{sec:grad:moderate}. The details are left to the reader. 
\begin{theorem}
  Let $p=0$.
  Operator $\PiF^\nabla = \PiFtilde^{\nabla,0}$ resp. $\PiF^\nabla = \PiF^{\nabla,0}$ is idempotent on $P^0(T)$, satisfies~\eqref{eq:fortinprop:a} and
  \begin{align*}
    \norm{\PiF^\nabla v}T \lesssim \norm{v}T + h_T\norm{\nabla v}T, \quad
    \norm{\nabla \PiF^\nabla v}T \lesssim \norm{\nabla v}T
  \end{align*}
  for all $v\in H^1(T)$. Operator $\PiF^\nabla = \PiF^{\nabla,0}$ additionally satisfies~\eqref{eq:fortinprop:b}.
  \qed
\end{theorem}

\subsection{Comparison with existing Fortin operators}\label{sec:fortin:existing}
As mentioned in the introduction, several works have already dealt with the construction of Fortin operators on a simplex that satisfy~\eqref{eq:fortinprop}, see, e.g.,~\cite{practicalDPG,breakSpace,DemkowiczZanotti20}.
These works all have in common that they construct resp. prove the existence of a Fortin operator $\PiF^\nabla\colon H^1(T) \to P^{p+n}(T)$ which satisfy~\eqref{eq:fortinprop:bound}-\eqref{eq:fortinprop:a} (for $\alpha\gtrsim h_T$) and
\begin{align*}
  \ip{u}{(1-\PiF^\nabla)v}_T = 0 \quad\forall u\in P^{p-1}(T), \, v\in H^1(T).
\end{align*}
The latter condition is not the same as~\eqref{eq:fortinprop:b}. In order to satisfy~\eqref{eq:fortinprop:b} one needs to increase the polynomial degree by one, i.e., $P^{p+1+n}(T)$ instead of $P^{p+n}(T)$.

Comparing the dimension of $P^{p+1+n}(T)$ and the space $V_{hp}^\nabla$, we get
\begin{align*}
  \dim(P^{p+1+n}(T)) &= \frac{\prod_{j=1}^n(p+1+n+j)}{n!}, \\
  \dim(V_{hp}^\nabla) &= 1 + (n+1)\frac{\prod_{j=1}^{n-1}(p+j)}{(n-1)!} + \frac{\prod_{j=1}^n(p+j)}{n!}.
\end{align*}
For the lowest-order case $p=0$ we thus find
\begin{align*}
  \dim(P^{n+1}(T)) &= \begin{cases}
    10 & n=2, \\
    35 & n=3,
  \end{cases}
  \quad\text{and}\quad
  \dim(V_{h0}^\nabla) = \begin{cases}
    5 & n=2, \\
    6 & n=3.
  \end{cases}
\end{align*}
For operator $\PiF^{\nabla,0}\colon H^1(T)\to V_h^{\nabla,0}$ we even have a reduction to
\begin{align*}
  \dim(V_{h}^{\nabla,0}) = n+1+1 =
  \begin{cases}
    4 & n=2, \\
    5 & n=3.
  \end{cases}
\end{align*}
In conclusion, our test spaces are systematically smaller than previously used ones, and guarantee robustness contrary to the previous cases. 

\section{Fortin operator in $\Hdivset{T}$}\label{sec:fortinDiv}
We consider a fixed parameter $\alpha>0$ and space $\Hdivset{T}$ equipped with the (squared) norm
\begin{align*}
  \norm{\ttau}{T,\alpha}^2 = \norm{\ttau}{T}^2 + \alpha^2 \norm{\div\ttau}{T}^2 \quad\text{for } \ttau\in\Hdivset{T}.
\end{align*}
The motivation for this section is the construction of Fortin operators, say $\PiF^{\div}\colon \Hdivset{T}\to V_h$ (where $V_h\subseteq\Hdivset{T}$ is some discrete space) such that, for all $\ttau\in \Hdivset{T}$,
\begin{subequations}\label{eq:fortinDiv}
\begin{align}
  \norm{\PiF^{\div}\ttau}{T,\alpha} &\leq C_\mathrm{F} \norm{\ttau}{T,\alpha}, \label{eq:fortinDiv:bound}\\
  \dual{(\ttau-\PiF^{\div}\ttau)\cdot\normal_T}{u}_{\partial T} &= 0 \quad\forall u\in P_c^{p+1}(\FF_T), \label{eq:fortinDiv:a}\\
  \ip{\ssigma}{\ttau-\PiF^{\div}\ttau}_{T} &= 0 \quad\forall \ssigma \in \PP^p(T) \label{eq:fortinDiv:b}
\end{align}
with $C_\mathrm{F}>0$ independent of $\alpha$, $h_T$ (but possibly dependent on $p\in\N_0$).
The latter two identities also imply 
\begin{align}
  \ip{u}{\div(\ttau-\PiF^{\div}\ttau)}_T = 0 \quad\forall u\in P^{p+1}(T), \label{eq:fortinDiv:c}
\end{align}
\end{subequations}
which can be seen from integration by parts: Let $u\in P^{p+1}(T)$ be given, then
\begin{align*}
  \ip{u}{\div\PiF^{\div}\ttau}_T &= -\ip{\nabla u}{\PiF^{\div}\ttau}_T + \dual{\PiF^{\div}\ttau\cdot\normal_T}{u|_{\partial T}}_{\partial T} 
  \\
  &=-\ip{\nabla u}\ttau_T + \dual{\ttau\cdot\normal_T}{u|_{\partial T}}_{\partial T}  = \ip{u}{\div\ttau}_T \quad\forall \ttau\in \Hdivset{T}.
\end{align*}

\subsection{Construction for moderate parameter}
Define 
\begin{align*}
  \widetilde V_{hp}^{\div} = \linhull\set{\ppsi_{\partial T,j}}{j=1,\dots,\dim(\widetilde P^{p+1}(T))} \subset \PP^{p+2}(T)
\end{align*}
and operator $\PiFptilde^{\div}\colon \Hdivset{T}\to \widetilde V_{hp}^{\div}$ by
\begin{align*}
  \PiFptilde^{\div}\ttau = \sum_{j=1}^{\dim(P_c^{p+1}(\partial T))} \frac{\dual{\ttau\cdot\normal_T}{\nu_{\partial T,j}}_{\partial T}}{\dual{\ppsi_{\partial T,j}\cdot\normal_T}{\nu_{\partial T,j}}_{\partial T}} \ppsi_{\partial T,j}.
\end{align*}
We collect its main properties.
\begin{lemma}\label{lem:fortinDiv:normal}
  Operator $\PiF^{\div} = \PiFptilde^{\div}$ satisfies~\eqref{eq:fortinDiv:a} and
  \begin{align*}
    \norm{\PiFptilde^{\div}\ttau}T &\lesssim \norm{\ttau}T + h_T\norm{\div\ttau}{T} \quad \forall \ttau\in\Hdivset{T}.
  \end{align*}
\end{lemma}
\begin{proof}
  To see~\eqref{eq:fortinDiv:a} a simple computation yields
  \begin{align*}
    \dual{(1-\PiFptilde^{\div})\ttau\cdot\normal_T}{\nu_{\partial T,k}}_{\partial T}
    &= \dual{\ttau\cdot\normal_T}{\nu_{\partial T,k}}_{\partial T} - 
    \sum_{j}\frac{\dual{\ttau\cdot\normal_T}{\nu_{\partial T,j}}_{\partial T}}{\dual{\ppsi_{\partial T,j}\cdot\normal_T}{\nu_{\partial T,j}}_{\partial T}} \dual{\ppsi_{\partial T,j}\cdot\normal_T}{\nu_{\partial T,k}}_{\partial T}
    \\
    &= \dual{\ttau\cdot\normal_T}{\nu_{\partial T,k}}_{\partial T} - \dual{\ttau\cdot\normal_T}{\nu_{\partial T,k}}_{\partial T} = 0.
  \end{align*}
  Resolving the duality term, the Cauchy--Schwarz inequality and properties of basis functions give
  \begin{align*}
    |\dual{\ttau\cdot\normal_T}{\nu_{\partial T,j}}_{\partial T}| &= |\ip{\ttau}{\nabla \nu_{\partial T,j}}_T + \ip{\div\ttau}{\nu_{\partial T,j}}_T| 
    \\
    &\lesssim \norm{\ttau}Th_T^{-1}|T|^{1/2} + \norm{\div\ttau}T|T|^{1/2}.
  \end{align*}
  With the triangle inequality we conclude that
  \begin{align*}
    \norm{\PiFptilde^{\div}\ttau}T &\leq \sum_{j} \frac{|\dual{\ttau\cdot\normal_T}{\nu_{\partial T,j}}_{\partial T}|}{\dual{\ppsi_{\partial T,j}\cdot\normal_T}{\nu_{\partial T,j}}_{\partial T}} \norm{\ppsi_{\partial T,j}}T
    \\
    &\lesssim |\partial T|^{-1} (\norm{\ttau}Th_T^{-1}|T|^{1/2} + \norm{\div\ttau}T|T|^{1/2})|T|^{1/2} \lesssim \norm{\ttau}T + h_T\norm{\div\ttau}T
  \end{align*}
  which finishes the proof.
\end{proof}

For the definition of operators that ensure~\eqref{eq:fortinDiv:b} we add a correction term based on the functions $\eeta_{E,j}$. We set
\begin{align*}
  V_{hp}^{\div} = \widetilde V_{hp}^{\div} + \linhull\set{\eeta_{E,j}}{j=1,\dots,\dim P^p(T),\, E\in \EE_*}
\end{align*}
and define $\PiFp^{\div}\colon \Hdivset{T}\to V_{hp}^{\div}$ for all $\ttau\in \Hdivset{T}$ by
\begin{align*}
  \PiFp^{\div}\ttau = \PiFptilde^{\div}\ttau + \sum_{E\in\EE_*} \sum_{j=1}^{\dim(P^p(T))} \frac{\ip{\ssigma_{E,j}}{(1-\PiFptilde^{\div})\ttau}_T}{\ip{\ssigma_{E,j}}{\eeta_{E,j}}_T} \eeta_{E,j}.
\end{align*}

\begin{theorem}\label{thm:fortinDiv:normal}
  Operator $\PiF^{\div} = \PiFp^{\div}$ satisfies~\eqref{eq:fortinDiv:a}--\eqref{eq:fortinDiv:b} and
  \begin{align*}
    \norm{\PiFp^{\div}\ttau}T \lesssim \norm{\ttau}T + h_T\norm{\div\ttau}T \quad\forall \ttau\in \Hdivset{T}.
  \end{align*}
  Furthermore, $\div\circ\PiFp^{\div} = \Pi_T^{p+1}\circ\div$.
  If $h_T\lesssim \alpha$, then $\PiF^{\div} = \PiFp^{\div}$ satisfies~\eqref{eq:fortinDiv:bound}.
\end{theorem}
\begin{proof}
  Since $\eeta_{E,j}\cdot\normal_T|_{\partial T} = 0$, we get $\PiFp^{\div}\ttau\cdot\normal_T|_{\partial T} =  \PiFptilde^{\div}\ttau\cdot\normal_T|_{\partial T}$, thus, condition~\eqref{eq:fortinDiv:a} follows from Lemma~\ref{lem:fortinDiv:normal}.
  To see~\eqref{eq:fortinDiv:b} we compute for any $E',k$
  \begin{align*}
    \ip{\ssigma_{E',k}}{(1-\PiFp^{\div})\ttau}_T &= \ip{\ssigma_{E',k}}{(1-\PiFptilde^{\div})\ttau}_T - \sum_{E\in\EE_*}\sum_{j} \frac{\ip{\ssigma_{E,j}}{(1-\PiFptilde^{\div})\ttau}_T}{\ip{\ssigma_{E,j}}{\eeta_{E,j}}_T} \ip{\ssigma_{E',k}}{\eeta_{E,j}}_T
    \\
    &= \ip{\ssigma_{E',k}}{(1-\PiFptilde^{\div})\ttau}_T -\ip{\ssigma_{E',k}}{(1-\PiFptilde^{\div})\ttau}_T =0.
  \end{align*}
  Noting that properties of the basis functions and boundedness by Lemma~\ref{lem:fortinDiv:normal} give
  \begin{align*}
    \frac{|\ip{\ssigma_{E,j}}{(1-\PiFptilde^{\div})\ttau}_T|}{\ip{\ssigma_{E,j}}{\eeta_{E,j}}_T} \norm{\eeta_{E,j}}T 
    \lesssim \norm{(1-\PiFptilde^{\div})\ttau}T \lesssim \norm{\ttau}T + h_T\norm{\div\ttau}T,
  \end{align*}
  we see that
  \begin{align*}
    \norm{\PiFp^{\div}\ttau}T \leq \norm{\PiFptilde^{\div}\ttau}T + \sum_{E\in\EE_*} \sum_{j} \frac{|\ip{\ssigma_{E,j}}{(1-\PiFptilde^{\div})\ttau}_T|}{\ip{\ssigma_{E,j}}{\eeta_{E,j}}_T} \norm{\eeta_{E,j}}T \lesssim \norm{\ttau}T + h_T\norm{\div\ttau}T.
  \end{align*}
  The commutativity property follows from~\eqref{eq:fortinDiv:c} and the fact that $V_{hp}^{\div}\subseteq \PP^{p+2}(T)$ and, therefore, $\div\PiFp^{\div}\ttau\in P^{p+1}(T)$.
  
  For the final assertion note that $\norm{\div\PiFp^{\div}\ttau}T=\norm{\Pi_T^{p+1}\div\ttau}T \leq \norm{\div\ttau}T$ by the commutativity property. Together with boundedness in the $L^2(T)$ norm established above, this finishes the proof.
\end{proof}

\begin{remark}\label{rem:supconv:div}
  In general, discrete test spaces are chosen such that a Fortin operator exists, which not necessarily implies approximation results of the form
  \begin{align*}
    \min_{\ttau_h\in V_h} \norm{\ttau-\ttau_h}{T} + \norm{\div(\ttau-\ttau_h)}T \lesssim h_T \norm{\nabla\ttau}{T} + \norm{(1-\Pi_T^0)\div\ttau}T \quad\text{for } \ttau \in \HH^1(T). 
  \end{align*}
  For certain supercloseness results in the DPG method the latter approximation property is required, see~\cite{SupConv2}.
  To ensure this property one can simply require $\RT^0(T)\subset V_h$.
\end{remark}

\subsection{Construction for small parameter}
For a small parameter, i.e., $\alpha\lesssim h_T$, we build a Fortin operator quite a bit different to the ones presented in the previous section. 
The proof of boundedness requires as in the scalar case the multiplicative version of the trace inequality, but also the following Helmholtz decomposition together with elliptic regularity.
\begin{lemma}\label{lem:helmholtz}
  Let $\ttau\in \Hdivset{T}$. There exist $r\in H_0^1(T)$, $\qq\in \HH(\ccurl;T)$ (for $n=2$ we have $\qq\in H^1(T)$)
  such that
  $\ttau = \nabla r + \ccurl\qq$  and 
  \begin{align*}
    \norm{\ccurl \qq}T^2 + \norm{\nabla r}T^2 = \norm{\ttau}{T}^2, \quad \norm{D^2r}{T} \lesssim \norm{\div\ttau}T
  \end{align*}
  with constants independent of $T$. 
\end{lemma}
\begin{proof}
  Define $r\in H_0^1(T)$ as the weak solution of $\Delta r = \div\ttau$, $r|_{\partial T} = 0$. Elliptic regularity implies (note that $\div\ttau\in L^2(T)$ and $T$ is convex) that $r\in H^2(T)$. Moreover,~\cite[Theorem~3.1.1.2]{grisvard} implies that 
  \begin{align*}
    \sum_{|\beta|=2} \norm{D^\beta r}{T} \leq C \norm{\Delta r}{T}
  \end{align*}
  with a constant independent of $T$.
  Finally, $\div(\ttau-\nabla r) = 0$ by construction implies that $\ttau-\nabla r = \ccurl \qq$ for some $\qq\in \HH(\ccurl;T)$ for $n=3$ resp. $\qq\in H^1(T)$ for $n=2$.
\end{proof}

The definition and analysis of our Fortin operator is based on the Helmholtz decomposition $\ttau = \nabla r + \ccurl\qq$. For the $\ccurl\qq$ contribution we use the Fortin operator defined in the previous section. For $\nabla r$ we consider the following definitions and analysis: Define the space
\begin{align*}
  \widetilde V_{hp,\alpha}^{\div,\mathrm{aux}} = \PP^0(T) + \linhull\set{\eeta_{\alpha,F,j}}{j=1,\dots,\dim(P_c^{p+1}(\partial T))}
\end{align*}
and operator $\PiFptilde^{\div,\mathrm{aux}}\colon \Hdivset{T}\to \widetilde V_{hp,\alpha}^{\div,\mathrm{aux}}$ for all $\ttau\in \Hdivset{T}$ by
\begin{align*}
  \PiFaptilde^{\div,\mathrm{aux}}\ttau &= \Pi_T^0\ttau + \sum_{j=1}^{\dim(P_c^{p+1}(\FF_T))} 
  \frac{\dual{(1-\Pi_T^0)\ttau\cdot\normal_T}{\chi_{\partial T,j}}_{\partial T}}{\dual{\eeta_{\alpha,\partial T,j}\cdot\normal_T}{\chi_{\partial T,j}}_{\partial T}} \eeta_{\alpha,\partial T,j}.
\end{align*}

\begin{lemma}\label{lem:fortinDivAux}
  Operator $\PiF^{\div} = \PiFaptilde^{\div,\mathrm{aux}}$ satisfies~\eqref{eq:fortinDiv:a}.
  If $\alpha\lesssim h_T$, then
  \begin{align*}
    \norm{\PiFaptilde^{\div,\mathrm{aux}}\nabla r}{T,\alpha} \lesssim \norm{\nabla r}{T,\alpha}
  \end{align*}
  for all $r\in H_0^1(T)\cap H^2(T)$.
\end{lemma}
\begin{proof}
  The verification of~\eqref{eq:fortinDiv:a} follows similarly as in Lemma~\ref{lem:fortinDiv:normal}. We leave the details to the reader and focus on details for the proof of boundedness.
  Let $r\in H_0^1(T)\cap H^2(T)$.
  Note that elliptic regularity yields $\norm{D^2r}{T} \lesssim \norm{\div\nabla r}{T}$ with constant independent of $T$. 
  Using standard norm estimates, the multiplicative version of the trace inequality (see Lemma~\ref{lem:traceineq}) and Lemma~\ref{lem:modRB:ho}, we get 
  \begin{align*}
    \frac{|\dual{(1-\Pi_T^0)\nabla r\cdot\normal_T}{\chi_{\partial T,j}}_{\partial T}|}{\dual{\eeta_{\alpha,\partial T,j}\cdot\normal_T}{\chi_{\partial T,j}}_{\partial T}} \norm{\eeta_{\alpha,\partial T,j}}T 
    &\lesssim \norm{(1-\Pi_T^0)\nabla r}{\partial T} |\partial T|^{-1/2}|T|^{1/2} (\alpha/h_T)^{1/2}
    \\
    &\lesssim  \norm{(1-\Pi_T^0)\nabla r}T^{1/2} \alpha^{1/2} \norm{D^2r}{T}^{1/2}
    \\
    &\lesssim \norm{(1-\Pi_T^0)\nabla r}T + \alpha\norm{\Delta r}{T} \leq \norm{\nabla r}T + \alpha\norm{\Delta r}T.
  \end{align*}
  By summing over all indices and bounding $\norm{\Pi_T^0\nabla r}T \leq \norm{\nabla r}T$ we find that
  \begin{align*}
    \norm{\PiFaptilde^{\div,\mathrm{aux}}\nabla r}T \lesssim \norm{\Pi_T^0\nabla r}T  + \norm{\nabla r}T + \alpha \norm{\Delta r}T \lesssim \norm{\nabla r}T + \alpha \norm{\div\nabla r}T.
  \end{align*}
  The same argumentation also proves
  \begin{align*}
    \alpha\frac{|\dual{(1-\Pi_T^0)\nabla r\cdot\normal_T}{\chi_{\partial T,j}}_{\partial T}|}{\dual{\eeta_{\alpha,\partial T,j}\cdot\normal_T}{\chi_{\partial T,j}}_{\partial T}} \norm{\div\eeta_{\alpha,\partial T,j}}T 
    &\lesssim \alpha \norm{(1-\Pi_T^0)\nabla r}{\partial T} |\partial T|^{-1/2}|T|^{1/2} h_T^{-1}(\alpha/h_T)^{-1/2}
    \\
    &\lesssim  \norm{(1-\Pi_T^0)\nabla r}T^{1/2} \alpha^{1/2} \norm{D^2r}{T}^{1/2}
    \\
    &\lesssim \norm{(1-\Pi_T^0)\nabla r}T + \alpha\norm{\div\nabla r}{T} 
    \\ &\leq \norm{\nabla r}T + \alpha\norm{\div\nabla r}T.
  \end{align*}
  Summing over all indices and using $\div\Pi_T^0\nabla r = 0$ we conclude that
  \begin{align*}
    \alpha\norm{\div\PiFaptilde^{\div,\mathrm{aux}}\nabla r}T \lesssim \norm{\nabla r}T + \alpha \norm{\div\nabla r}T.
  \end{align*}
  Combining all estimates finishes the proof.
\end{proof}

For the definition of operators that ensure~\eqref{eq:fortinDiv:b} we add --- as before --- a correction term based on the functions $\eeta_{E,j}$. We stress that these edge functions do not need to be modified. Set
\begin{align*}
  V_{hp,\alpha}^{\div} = V_{hp}^{\div} + \widetilde V_{hp}^{\div,\mathrm{aux}}
\end{align*}
and define $\PiFp^{\div}\colon \Hdivset{T}\to V_{hp}^{\div}$ for all $\ttau = \nabla r + \ccurl\qq\in \Hdivset{T}$ by
\begin{align*}
  \PiFap^{\div}\ttau = \PiFp^{\div}\ccurl\qq + \PiFaptilde^{\div,\mathrm{aux}}\nabla r + \sum_{E\in\EE_*} \sum_{j=1}^{\dim(P^p(T))} \frac{\ip{\ssigma_{E,j}}{(1-\PiFaptilde^{\div,\mathrm{aux}})\nabla r}_T}{\ip{\ssigma_{E,j}}{\eeta_{E,j}}_T} \eeta_{E,j}.
\end{align*}

The following is one of our main results.
\begin{theorem}\label{thm:fortinDiv}
  Operator $\PiF^{\div} = \PiFap^{\div}$ satisfies~\eqref{eq:fortinDiv:a}--\eqref{eq:fortinDiv:b}.
  If $\alpha\lesssim h_T$, then
  \begin{align*}
    \norm{\PiFap^{\div}\ttau}{T,\alpha} \lesssim \norm{\ttau}{T,\alpha} \quad\forall \ttau\in \Hdivset{T}.
  \end{align*}
\end{theorem}
\begin{proof}
  The verification of~\eqref{eq:fortinDiv:a}--\eqref{eq:fortinDiv:b} follows from the arguments already seen in Theorem~\ref{thm:fortinDiv:normal} and Lemma~\ref{lem:fortinDivAux}. 

  It only remains to prove boundedness.
  By the triangle inequality we get
  \begin{align*}
    \norm{\PiFap^{\div}\ttau}{T,\alpha} \leq \norm{\PiFap^{\div}\ccurl\qq}{T,\alpha} + \norm{\PiFap^{\div}\nabla r}{T,\alpha}
  \end{align*}
  for $\ttau=\nabla r + \ccurl \qq\in \Hdivset{T}$ where $r,\qq$ are defined as in Lemma~\ref{lem:helmholtz}.

  Note that $\PiFap^{\div}\ccurl\qq = \PiFp^{\div}\ccurl\qq$ and applying Theorem~\ref{thm:fortinDiv:normal} and Lemma~\ref{lem:helmholtz} we see that
  \begin{align*}
    \norm{\PiFp^{\div}\ccurl\qq}T \lesssim \norm{\ccurl\qq}T + h_T\norm{\div\ccurl\qq}T = \norm{\ccurl\qq}T \leq \norm{\ttau}T
  \end{align*}
  as well as $\div\PiFp^{\div}\ccurl\qq = \Pi_T^{p+1}\div\ccurl\qq = 0$. We conclude that
  \begin{align*}
  \norm{\PiFap^{\div}\ccurl\qq}{T,\alpha} \lesssim \norm{\ttau}T \leq \norm{\ttau}{T,\alpha}.
  \end{align*}
  It remains to estimate $\norm{\PiFap^{\div}\nabla r}{T,\alpha}$. Using the triangle inequality, the Cauchy--Schwarz inequality, estimates for basis functions and Lemma~\ref{lem:fortinDivAux}, we get
  \begin{align*}
    \norm{\PiFap^{\div}\nabla r}T &\leq \norm{\PiFaptilde^{\div,\mathrm{aux}}\nabla r}T + \sum_{E\in\EE_*}\sum_{j} \frac{|\ip{\ssigma_{E,j}}{(1-\PiFaptilde^{\div,\mathrm{aux}})\nabla r}_T|}{\ip{\ssigma_{E,j}}{\eeta_{E,j}}_T} \norm{\eeta_{E,j}}T
    \\
    & \lesssim \norm{\PiFaptilde^{\div,\mathrm{aux}}\nabla r}T +  |T|^{1/2}\norm{(1-\PiFaptilde^{\div,\mathrm{aux}})\nabla r}T |T|^{-1}|T|^{1/2}
    \lesssim \norm{\nabla r}{T,\alpha}.
  \end{align*}
  For the divergence part in the norm, recall that $\eeta_{E,j} \in \PP^{p+2}(T)$. Therefore, $\div\eeta_{E,j}\in P^{p+1}(T)$ and~\eqref{eq:fortinDiv:c} implies
  \begin{align*}
    \ip{\div\eeta_{E,j}}{\div(1-\PiFap^{\div})\nabla r}_T = 0, \quad j=1,\dots,\dim(P^p(T)), \, E\in\EE_*. 
  \end{align*}
  We conclude that
  \begin{align*}
    \norm{\div(1-\PiFap^{\div})\nabla r}T^2 &= \ip{\div(1-\PiFap^{\div})\nabla r}{\div(1-\PiFap^{\div})\nabla r}_T
    \\
    &= \ip{\div(1-\PiFaptilde^{\div,\mathrm{aux}})\nabla r}{\div(1-\PiFap^{\div})\nabla r}_T
    \\
    &\leq \norm{\div(1-\PiFaptilde^{\div,\mathrm{aux}})\nabla r}T\norm{\div(1-\PiFap^{\div})\nabla r}T.
  \end{align*}
  It follows that $\norm{\div(1-\PiFap^{\div})\nabla r}T\leq \norm{\div(1-\PiFaptilde^{\div,\mathrm{aux}})\nabla r}T$ and with the triangle inequality and Lemma~\ref{lem:fortinDivAux} we get that
  \begin{align*}
    \alpha\norm{\div\PiFap^{\div}\nabla r}T \leq \alpha\norm{\div\nabla r}T + \alpha\norm{\div(1-\PiFaptilde^{\div,\mathrm{aux}})\nabla r}T \lesssim \norm{\nabla r}{T,\alpha}
  \end{align*}
  which finishes the proof together with $\norm{\nabla r}{T,\alpha}\leq \norm{\ttau}{T,\alpha}$ (Lemma~\ref{lem:helmholtz}).
\end{proof}

\subsection{Alternative operator for lowest order and moderate parameter}\label{sec:fortinDiv:lo:normal}
First, consider the spaces 
\begin{align*}
  \widetilde V_h^{\div,1} = \RT^0(T), \quad \widetilde V_h^{\div,2} = \PP^0(T) + \linhull\set{\eeta_F}{F\in\FF_T}
\end{align*}
and operators $\PiFtilde^{\div,j}\colon \Hdivset{T}\to \widetilde V_h^{\div,j}$, $j=1,2$, for all $\ttau\in\Hdivset{T}$ by
\begin{subequations}\label{eq:def:PiDivTilde}
\begin{align}
  \PiFtilde^{\div,1}\ttau &= \sum_{F\in\FF_T} \frac{\dual{\ttau\cdot\normal_T}{\nu_F}_{\partial T}}{\dual{\ppsi_F\cdot\normal_T}{\nu_F}_{\partial T}} \ppsi_F,
  \\
  \PiFtilde^{\div,2}\ttau &= \Pi_T^0\ttau + \sum_{F\in\FF_T} \frac{\dual{(1-\Pi_T^0)\ttau\cdot\normal_T}{\nu_F}_{\partial T}}{\dual{\eeta_F\cdot\normal_T}{\nu_F}_{\partial T}} \eeta_F.
\end{align}
\end{subequations}
One verifies that $\PiFtilde^{\div,1}$ is a projector whereas $\PiFtilde^{\div,2}$ is idempotent on $\PP^0(T)$.
Defining the spaces
\begin{align*}
  V_h^{\div,j} = \widetilde V_h^{\div,j} + \linhull\set{\eeta_E}{E\in\EE_*}, \quad j=1,2,
\end{align*}
we introduce operators $\PiF^{\div,j}\colon V\to V_h^{\div,j}$, $j=1,2$, for all $\ttau\in\Hdivset{T}$ by
\begin{align*}
  \PiF^{\div,j} \ttau = \PiFtilde^{\div,j}\ttau + \sum_{E\in\EE_*} \frac{\ip{\ssigma_E}{(1-\PiFtilde^{\div,j})\ttau}_T}{\ip{\ssigma_E}{\eeta_E}_T} \eeta_E.
\end{align*}

\begin{theorem}\label{thm:fortinDiv:normal:lo}
  Operator $\PiF^{\div}\in\set{\PiF^{\div,j}}{j=1,2}$ satisfies~\eqref{eq:fortinDiv:a}--\eqref{eq:fortinDiv:b}
  and
  \begin{align*}
    \norm{\PiF^{\div}\ttau}T &\lesssim \norm{\ttau}T + h_T \norm{\div\ttau}T, \quad
    \norm{\div\PiF^{\div}\ttau}T \lesssim \norm{\div\ttau}T \quad\forall \ttau\in \Hdivset{T}.
  \end{align*}
  Moreover, $\div\circ \PiF^{\div,1} = \Pi_T^1\circ \div$.
  Furthermore, operator $\PiF^{\div,1}$ is idempotent on $\RT^0(T)$ and $\PiF^{\div,2}$ is idempotent on $\PP^0(T)$.
\end{theorem}
\begin{proof}
  Idempotency of the operators follows from their definitions. 
  Let $\PiF^{\div}\in\set{\PiF^{\div,j}}{j=1,2}$ and $\PiFtilde^{\div}\in\set{\PiFtilde^{\div,j}}{j=1,2}$.
  First, we check condition~\eqref{eq:fortinDiv:a}. It holds for $\PiFtilde^{\div}$ as can be seen with the same arguments as in Lemma~\ref{lem:fortinDiv:normal}. Since $\eeta_E\cdot\normal_T|_{\partial T} = 0$ by Lemma~\ref{lem:rotatedBR} we have $\PiF^{\div}\ttau\cdot\normal_T|_{\partial T} = \PiFtilde^{\div}\ttau\cdot\normal_T|_{\partial T}$. We conclude that~\eqref{eq:fortinDiv:a} follows.

  Second, condition~\eqref{eq:fortinDiv:b} can be seen as follows,
  \begin{align*}
    \ip{\ssigma_{E'}}{\PiF^{\div}\ttau}_T &= \ip{\ssigma_{E'}}{\PiFtilde^{\div}\ttau}_T +
    \sum_{E\in\EE_*} \frac{\ip{\ssigma_E}{(1-\PiFtilde^{\div})\ttau}_T}{\ip{\ssigma_E}{\eeta_E}_T} \ip{\ssigma_{E'}}{\eeta_E}_T
    \\
    &= \ip{\ssigma_{E'}}{\PiFtilde^{\div}\ttau}_T + \ip{\ssigma_{E'}}{(1-\PiFtilde^{\div})\ttau}_T = \ip{\ssigma_{E'}}{\ttau}_T
    \quad\forall E'\in\EE_*.
  \end{align*}
  Next we prove boundedness. Estimate $\norm{\PiFtilde^{\div,1}\ttau}T\lesssim \norm{\ttau}T + h_T\norm{\div\ttau}T$ follows as in Lemma~\ref{lem:fortinDiv:normal}. 
  For the second operator we stress that
  \begin{align*}
    |\dual{(1-\Pi_T^0)\ttau\cdot\normal_T}{\nu_F}_T| = |\ip{\div\ttau}{\nu_F}_T + \ip{(1-\Pi_T^0)\ttau}{\nabla \nu_F}_T|
    = |\ip{\div\ttau}{\nu_F}_T| \lesssim \norm{\div\ttau}T|T|^{1/2}.
  \end{align*}
  The second identity follows since $\nabla\nu_F\in\PP^0(T)$.
  This allows us to estimate
  \begin{align*}
    \norm{\PiFtilde^{\div,2}\ttau}T \lesssim \norm{\Pi_T^0\ttau}T + \sum_{F\in\FF_T} |F|^{-1}|T| \norm{\div\ttau}T \lesssim \norm{\ttau}T + h_T\norm{\div\ttau}T.
  \end{align*}
  Then, by the triangle inequality, the Cauchy--Schwarz inequality, norm estimates of basis functions and the previously established boundedness estimates, we see that
  \begin{align*}
    \norm{\PiF^{\div}\ttau}T &\leq \norm{\PiFtilde^{\div}\ttau}T + 
    \sum_{E\in\EE_*} \frac{|\ip{\ssigma_E}{(1-\PiFtilde^{\div})\ttau}_T|}{\ip{\ssigma_E}{\eeta_E}_T} \norm{\eeta_E}T
    \\
    &\lesssim \norm{\PiFtilde^{\div}\ttau}T + |T|^{1/2}\norm{(1-\PiFtilde^{\div})\ttau}T|T|^{-1}|T|^{1/2}
    \\
    &\lesssim \norm{\PiFtilde^{\div}\ttau}T + \norm{(1-\PiFtilde^{\div})\ttau}T  \lesssim \norm{\ttau}T + h_T\norm{\div\ttau}T.
  \end{align*}
  Furthermore, the same arguments and $\div\Pi_T^0\ttau=0$, $\norm{\div\eeta_F}T\eqsim h_T^{-1}|T|^{1/2}$ show that
  \begin{align*}
    \norm{\div\PiFtilde^{\div,2}\ttau}T \lesssim \norm{\div\ttau}T.
  \end{align*}
  Note that $\PiF^{\div,1}\ttau\in \PP^2(T)$, thus, $\div\PiF^{\div,1}\ttau\in P^1(T)$. Consequently,~\eqref{eq:fortinDiv:c} implies the commutativity property of $\PiF^{\div,1}$.
  Clearly, this also yields $\norm{\div\PiF^{\div,1}\ttau}T\leq \norm{\div\ttau}T$. It thus remains to prove that $\norm{\div\PiF^{\div,2}\ttau}T \lesssim \norm{\div\ttau}T$.
  To do so we argue as in the proof of Theorem~\ref{thm:fortinDiv} to derive $\norm{\div(1-\PiF^{\div,2})\ttau}T\leq \norm{\div(1-\PiFtilde^{\div,2})\ttau}T$. Together with the triangle inequality and $\norm{\div\PiFtilde^{\div,2}\ttau}T \lesssim \norm{\div\ttau}T$ we get that
  \begin{align*}
    \norm{\div\PiF^{\div,2}\ttau}T \leq \norm{\div\ttau}T + \norm{\div(1-\PiFtilde^{\div,2})\ttau}T \lesssim \norm{\div\ttau}T.
  \end{align*}
  This finishes the proof.
\end{proof}

\subsection{Alternative operator for lowest order and small parameter}\label{sec:fortinDiv:lo}
In this section we construct a simpler Fortin operator for $\alpha\lesssim h_T$ and the lowest-order case ($p=0$ in~\eqref{eq:fortinDiv}), based on $\PiF^{\div,2}$ from the previous section. 
Define the spaces
\begin{align*}
  \widetilde V_{h,\alpha}^{\div} &= \PP^0(T) + \linhull\set{\eeta_{\alpha,F}}{F\in\FF_T}, \\
  V_{h,\alpha}^{\div} &= \widetilde V_{h,\alpha}^{\div} + \linhull\set{\eeta_E}{E\in\EE_*}
\end{align*}
and operators $\PiFatilde^{\div}\colon \Hdivset{T}\to \widetilde V_{h,\alpha}^{\div}$, 
$\PiFa^{\div}\colon \Hdivset{T}\to V_{h,\alpha}^{\div}$ for all $\ttau\in \Hdivset{T}$ by
\begin{align*}
  \PiFatilde^{\div}\ttau &= \Pi_T^0\ttau + \sum_{F\in\FF_T} \frac{\dual{(1-\Pi_T^0)\ttau\cdot\normal_T}{\nu_F}_{\partial T}}{\dual{\eeta_{\alpha,F}\cdot\normal_T}{\nu_F}_{\partial T}} \eeta_{\alpha,F}, \\
  \PiFa^{\div}\ttau &= \PiFatilde^{\div}\ttau + \sum_{E\in\EE_*} \frac{\ip{\ssigma_E}{(1-\PiFatilde^{\div})\ttau}_T}{\ip{\ssigma_E}{\eeta_E}_T} \eeta_E. 
\end{align*}

\begin{theorem}\label{thm:fortinDiv:lo}
  Operator $\PiF^{\div} = \PiFa^{\div}$ satisfies~\eqref{eq:fortinDiv:a}--\eqref{eq:fortinDiv:b} for $p=0$. 
  If $\alpha\lesssim h_T$ then
  \begin{align*}
    \norm{\PiFa^{\div}\ttau}{T,\alpha} \lesssim  \norm{\ttau}{T,\alpha} \quad\forall \ttau\in\Hdivset{T}.
  \end{align*}
\end{theorem}
\begin{proof}
  The verification of~\eqref{eq:fortinDiv:a}--\eqref{eq:fortinDiv:b} follows as in Theorem~\ref{thm:fortinDiv:normal}. 
  
  It remains to prove boundedness. First, we show boundedness of $\PiFatilde^{\div}$. Let $\ttau = \nabla r + \ccurl\qq\in \Hdivset{T}$ with $r,\qq$ as in Lemma~\ref{lem:helmholtz} be given. Using $\nabla\nu_F\in \PP^0(T)$, $\div\ccurl\qq = 0$ we obtain
  \begin{align*}
    \PiFatilde^{\div}\ccurl\qq &= \Pi_T^0\ccurl\qq + \sum_{F\in\FF_T}\frac{\ip{\div(1-\Pi_T^0)\ccurl\qq}{\nu_F}_T + \ip{(1-\Pi_T^0)\ccurl\qq}{\nabla \nu_F}_T}{\dual{\eeta_{\alpha,F}\cdot\normal_T}{\nu_F}_{\partial T}}\eeta_{\alpha,F}
    \\&= \Pi_T^0 \ccurl\qq.
  \end{align*}
  Then, $\norm{\PiFatilde^{\div}\ccurl\qq}{T,\alpha}= \norm{\Pi_T^0\ccurl\qq}{T} \leq \norm{\ccurl\qq}T \leq \norm{\ttau}{T,\alpha}$.
  With the multiplicative trace inequality (Lemma~\ref{lem:traceineq}), properties of the basis functions, Lemma~\ref{lem:modRB} and Lemma~\ref{lem:helmholtz} we get 
  \begin{align*}
    \frac{|\dual{(1-\Pi_T^0)\nabla r\cdot\normal_T}{\nu_F}_{\partial T}|}{\dual{\eeta_{\alpha,F}\cdot\normal_T}{\nu_F}_{\partial T}}\norm{\eeta_{\alpha,F}}T 
    &\lesssim |F|^{-1}|\partial T|^{1/2} \norm{(1-\Pi_T^0)\nabla r}{\partial T} \norm{\eeta_{\alpha,F}}T
    \\
    &\lesssim |\partial T|^{-1/2}\norm{(1-\Pi_T^0)\nabla r}T^{1/2}\norm{D^2r}T^{1/2}\alpha^{1/2}h_T^{-1/2}|T|^{1/2}
    \\
    &\lesssim \norm{\nabla r}{T}^{1/2}\alpha^{1/2}\norm{\div\ttau}T^{1/2} \lesssim \norm{\ttau}T + \alpha\norm{\div\ttau}T.
  \end{align*}
  The last estimates yield $\norm{\PiFatilde^{\div}\nabla r}T\lesssim \norm{\ttau}{T,\alpha}$. For the divergence contribution the same arguments prove
  \begin{align*}
    \alpha\frac{|\dual{(1-\Pi_T^0)\nabla r\cdot\normal_T}{\nu_F}_{\partial T}|}{\dual{\eeta_{\alpha,F}\cdot\normal_T}{\nu_F}_{\partial T}}\norm{\div\eeta_{\alpha,F}}T 
    &\lesssim |F|^{-1}|\partial T|^{1/2} \norm{(1-\Pi_T^0)\nabla r}{\partial T} h_T^{-1}|T|^{1/2}\alpha^{1/2}h_T^{1/2} \\
    &\lesssim \norm{\nabla r}{T}^{1/2}\alpha^{1/2}\norm{\div\ttau}T^{1/2} \lesssim \norm{\ttau}T + \alpha\norm{\div\ttau}T.
  \end{align*}
  We conclude that $\alpha\norm{\div\PiFatilde^{\div}\nabla r}T \lesssim \norm{\ttau}{T,\alpha}$. 
  Putting all estimates together we have shown that
  \begin{align*}
    \norm{\PiFatilde^{\div}\ttau}{T,\alpha} \leq \norm{\PiFatilde^{\div}\ccurl\qq}{T,\alpha} + \norm{\PiFatilde^{\div}\nabla r}{T,\alpha}
    \lesssim \norm{\ttau}{T,\alpha}. 
  \end{align*}
  Arguing as in the proof of Theorem~\ref{thm:fortinDiv:normal:lo} we find that 
  \begin{align*}
    \norm{\PiFa^{\div}\ttau}T \lesssim \norm{\PiFatilde^{\div}\ttau}T + \norm{(1-\PiFatilde^{\div})\ttau}T
  \end{align*}
  giving us $\norm{\PiFa^{\div}\ttau}T \lesssim \norm{\ttau}{T,\alpha}$.
  Finally, arguing as in the proof of Theorem~\ref{thm:fortinDiv} we find that 
  \begin{align*}
    \norm{\div(1-\PiFa^{\div})\ttau}T \leq \norm{\div(1-\PiFatilde^{\div})\ttau}{T}.
  \end{align*}
  Together with the boundedness of $\PiFatilde^{\div}$ we conclude that $\alpha\norm{\div\PiFa^{\div}\ttau}T \lesssim \norm{\ttau}{T,\alpha}$ which finishes the proof.
\end{proof}

\subsection{Comparison with existing Fortin operators}
Fortin operators that satisfy~\eqref{eq:fortinDiv} are constructed in~\cite{practicalDPG,breakSpace,DemkowiczZanotti20}. 
In~\cite[Lemma~3.3]{practicalDPG}, it is shown that there exists a Fortin operator, mapping into the discrete test space $\PP^{p+2}(T)$ and in~\cite{DemkowiczZanotti20}, the authors impose the minimal condition $\RT^{p+1}(T)\subset V_h^{\div}$ to ensure the existence of a Fortin operator satisfying~\eqref{eq:fortinDiv}. 
Here, $\RT^{p+1}(T) = \PP^{p+1}(T)+\xx P^{p+1}(T)$ denotes the Raviart--Thomas space.
We stress that all these mentioned operators are uniformly bounded only if $h_T\lesssim \alpha$.

Computing dimensions we get
\begin{align*}
  \dim(\PP^{p+2}(T)) &= n \frac{\prod_{j=1}^n(j+p+2)}{n!}, \\
  \dim(\RT^{p+1}(T)) &= n\frac{\prod_{j=1}^n(j+p)}{n!} + (n+1)\frac{\prod_{j=1}^{n-1}(j+p+1)}{(n-1)!}.
\end{align*}
To compute the dimension of $V_{hp}^{\div}$ we note that $\dim(P_c^{p+1}(\FF_T)) = \dim(P^{p+1}(T))-\dim(P_b^{p+1}(T))$ where we recall that $P_b^p(T)$ denotes the space of element bubbles. 
This yields
\begin{align*}
  \dim(V_{hp}^{\div}) &= \dim(P_c^{p+1}(\FF_T))+\dim(\PP^p(T)) \\
  &= \frac{\prod_{j=1}^n (j+p+1)}{n!} - \frac{\prod_{j=1}^n(j+p-n)}{n!} + n\frac{\prod_{j=1}^n(j+p)}{n!}.
\end{align*}
For the lowest-order case $p=0$ we thus get
\begin{align*}
  \dim(\PP^2(T)) = \begin{cases}
    12 & n=2, \\
    30 & n=3,
  \end{cases}\quad
  \dim(\RT^1(T)) = \begin{cases}
    8 & n=2, \\
    15 & n=3,
  \end{cases}
\end{align*}
and
\begin{align*}
  \dim(V_{h0}^{\div}) = \begin{cases}
    5 & n=2, \\
    7 & n=3.
  \end{cases}
\end{align*}
As in Section~\ref{sec:fortin:existing} we conclude that our test spaces are systematically smaller than previously used ones, and guarantee robustness, contrary to the previous cases. 

\section{Numerical experiment}\label{sec:num}
In this section we consider the reaction-diffusion problem
\begin{align*}
  -\varepsilon^2\Delta u + u = f \quad\text{in } \Omega, \quad u|_{\partial\Omega} = 0. 
\end{align*}
First, we give a brief overview of a DPG method for the latter problem and, then, discuss results of our numerical experiment. 

\subsection{DPG method for reaction-diffusion problem}\label{sec:DPG}
We introduce the trace operators 
\begin{align*}
  \trgrad\colon H^1(\Omega) \to (\Hdivset{\TT})', \quad
  \trdiv\colon \Hdivset\Omega \to (H^1(\TT))',
\end{align*}
defined for $u\in H^1(\Omega)$, $\ssigma\in\Hdivset\Omega$ by 
\begin{alignat*}{2}
  \dual{\trgrad u}{\ttau}_{\partial\TT} &\,=\,& \ip{u}{\div_\TT\ttau}_\Omega + \ip{\nabla u}{\ttau}_\Omega &\quad\forall \ttau\in \Hdivset\TT = \prod_{T\in\TT}\Hdivset{T}, \\
  \dual{\trdiv\ssigma}v_{\partial\TT} &\,=\,& \ip{\ssigma}{\nabla_\TT v}_\Omega + \ip{\div\ssigma}{v}_\Omega &\quad\forall v\in H^1(\TT)=\prod_{T\in\TT} H^1(T).
\end{alignat*}
Here, $\div_\TT\colon \Hdivset{\TT}\to L^2(\Omega)$ is given by $\div_\TT\ttau|_T = \div(\ttau|_T)$ for $T\in\TT$, $\ttau=(\ttau_T)_{T\in\TT}\in\Hdivset\TT$ and $\nabla_\TT\colon H^1(\TT)\to\LL^2(\Omega)$ is defined similarly.
The trace spaces $H_{00}^{1/2}(\partial\TT) = \trgrad(H_0^1(\Omega))$, $H^{-1/2}(\partial\TT) = \trdiv(\Hdivset\Omega)$ are closed with respect to the canonical norms (see~\cite{breakSpace})
\begin{align*}
  \norm{\widehat u}{1/2,\varepsilon} &= \inf\set{\norm{u}{\Omega,\varepsilon}}{u\in H^1(\Omega), \,\trgrad  u  =\widehat u}, \\
  \norm{\widehat \sigma}{-1/2,\varepsilon} &= \inf\set{\norm{\ssigma}{\Omega,\varepsilon}}{\ssigma\in\Hdivset\Omega, \,\trdiv  \ssigma  =\widehat\sigma}.
\end{align*}
Introducing the spaces
\begin{align*}
  U &= L^2(\Omega)\times \LL^2(\Omega)\times H_{00}^{1/2}(\partial\TT)\times H^{-1/2}(\partial\TT), \\
  V &= H^1(\TT)\times \Hdivset\TT, 
\end{align*}
where $U$ is equipped with the canonical product norm and $V$ with the (squared) norm 
\begin{align*}
  \norm{(v,\ttau)}V^2 = \norm{v}{\TT,\varepsilon}^2 + \norm{\ttau}{\TT,\varepsilon}^2 :=
  \sum_{T\in\TT} \norm{v}{T,\varepsilon}^2 + \norm{\ttau}{T,\varepsilon}^2
\end{align*}
we obtain the ultraweak formulation of the reaction-diffusion problem by defining $\ssigma = \varepsilon\nabla u$ and element-wise integration by parts. This yields
\begin{align}\label{eq:DPG}
  \uu=(u,\ssigma,\widehat u,\widehat\sigma)\in U\colon 
  \qquad b(\uu,\vv) = L(\vv) \quad\forall \vv=(v,\ttau)\in V, 
\end{align}
where for $\uu\in U$, $\vv\in V$ and given $f\in L^2(\Omega)$ we define
\begin{align*}
  b(\uu,\vv) &= \ip{u}{\varepsilon\div_\TT\ttau+v}_\Omega + \ip{\ssigma}{\varepsilon\nabla_\TT v+ \ttau}_\Omega - \varepsilon\dual{\widehat u}{\ttau}_{\partial \TT} -\varepsilon\dual{\widehat\sigma}{v}_{\partial \TT}, \\
  L(\vv) &= \ip{f}{v}_\Omega.
\end{align*}
The following result contains well-posedness of the ultraweak formulation~\eqref{eq:DPG}. 
It can be derived by following the abstract theory presented in~\cite{breakSpace} together with our discussions on the fully-discrete scheme~\eqref{eq:DPG:practical} from the introduction.
\begin{proposition}\label{prop:dpg}
  The ultraweak formulation~\eqref{eq:DPG} admits a unique solution $\uu\in U$ with $\norm{\uu}U\lesssim \norm{f}\Omega$.
  Let $U_h\subset U$, $V_h\subset V$ denote finite-dimensional subspaces and suppose that there exists a Fortin operator $\PiF\colon V\to V_h$ satisfying~\eqref{eq:fortin:abstract}. 
  Then, with $\uu_h\in U$ denoting the solution of~\eqref{eq:DPG:practical},
  \begin{align*}
    \norm{\uu-\uu_h}U \lesssim \min_{\vv\in U_h} \norm{\uu-\vv_h}{U}.
  \end{align*}
  The hidden constants are independent of $\varepsilon$ and the mesh-size.
\end{proposition}

\subsection{Results for reaction-diffusion problem}
In this section we consider the manufactured solution
\begin{align*}
  u(x,y) = v(x)v(y) \quad\text{for } (x,y)\in \Omega = (0,1)^2,
\end{align*}
where 
\begin{align*}
  v(x) = 1-(1-e^{-1/(\sqrt{2}\varepsilon)})\frac{e^{-(1-x)/(\sqrt2\varepsilon)}+e^{-x/(\sqrt2\varepsilon)}}{1-e^{-2/(\sqrt{2}\varepsilon)}}.
\end{align*}
One verifies that $u$ is the solution of 
\begin{align*}
  -\varepsilon^2 \Delta u + u = f, \quad u|_{\partial \Omega} = 0 \quad\text{with } f(x,y) = \tfrac12(v(x)+v(y)).
\end{align*}

\begin{figure}
  \begin{center}
    \begin{tikzpicture}
\begin{loglogaxis}[
    title={$V_{h,\mathrm{pol}},\,\varepsilon=10^{-3}$},
width=0.49\textwidth,
cycle list/Dark2-6,
cycle multiindex* list={
mark list*\nextlist
Dark2-6\nextlist},
every axis plot/.append style={ultra thick},
xlabel={degrees of freedom},
ymin=1e-3, ymax=1,
grid=major,
legend entries={\tiny $\estDPG$,\tiny $\|u-u_h\|$,\tiny $\|\ssigma-\ssigma_h\|$},
legend pos=south west,
]
\addplot table [x=dof,y=est] {data/Vhpol1.dat};
\addplot table [x=dof,y=errUL2] {data/Vhpol1.dat};
\addplot table [x=dof,y=errSigma] {data/Vhpol1.dat};

\end{loglogaxis}
\end{tikzpicture}
\begin{tikzpicture}
\begin{loglogaxis}[
    title={$V_{h,\varepsilon},\,\varepsilon=10^{-3}$},
width=0.49\textwidth,
cycle list/Dark2-6,
cycle multiindex* list={
mark list*\nextlist
Dark2-6\nextlist},
every axis plot/.append style={ultra thick},
xlabel={degrees of freedom},
ymin=1e-3, ymax=1,
grid=major,
legend entries={\tiny $\estDPG$,\tiny $\|u-u_h\|$,\tiny $\|\ssigma-\ssigma_h\|$},
legend pos=south west,
]
\addplot table [x=dof,y=est] {data/Vhmod1.dat};
\addplot table [x=dof,y=errUL2] {data/Vhmod1.dat};
\addplot table [x=dof,y=errSigma] {data/Vhmod1.dat};

\end{loglogaxis}
\end{tikzpicture}
\begin{tikzpicture}
\begin{loglogaxis}[
    title={$V_{h,\mathrm{pol}},\,\varepsilon=10^{-4}$},
width=0.49\textwidth,
cycle list/Dark2-6,
cycle multiindex* list={
mark list*\nextlist
Dark2-6\nextlist},
every axis plot/.append style={ultra thick},
xlabel={degrees of freedom},
ymin=8e-4, ymax=1e-1,
grid=major,
legend entries={\tiny $\estDPG$,\tiny $\|u-u_h\|$,\tiny $\|\ssigma-\ssigma_h\|$},
legend pos=south east,
]
\addplot table [x=dof,y=est] {data/Vhpol2.dat};
\addplot table [x=dof,y=errUL2] {data/Vhpol2.dat};
\addplot table [x=dof,y=errSigma] {data/Vhpol2.dat};

\end{loglogaxis}
\end{tikzpicture}
\begin{tikzpicture}
\begin{loglogaxis}[
    title={$V_{h,\varepsilon},\,\varepsilon=10^{-4}$},
width=0.49\textwidth,
cycle list/Dark2-6,
cycle multiindex* list={
mark list*\nextlist
Dark2-6\nextlist},
every axis plot/.append style={ultra thick},
xlabel={degrees of freedom},
ymin=8e-4, ymax=1e-1,
grid=major,
legend entries={\tiny $\estDPG$,\tiny $\|u-u_h\|$,\tiny $\|\ssigma-\ssigma_h\|$},
legend pos=south west,
]
\addplot table [x=dof,y=est] {data/Vhmod2.dat};
\addplot table [x=dof,y=errUL2] {data/Vhmod2.dat};
\addplot table [x=dof,y=errSigma] {data/Vhmod2.dat};

\end{loglogaxis}
\end{tikzpicture}
  \end{center}
  \caption{Errors in the field variables compared with DPG estimator for $\varepsilon=10^{-3}$ and $\varepsilon=10^{-4}$. The left resp. right column shows the results using test space $V_{h,\mathrm{pol}}$ resp. $V_{h,\varepsilon}$.}
  \label{fig:comp}
\end{figure}
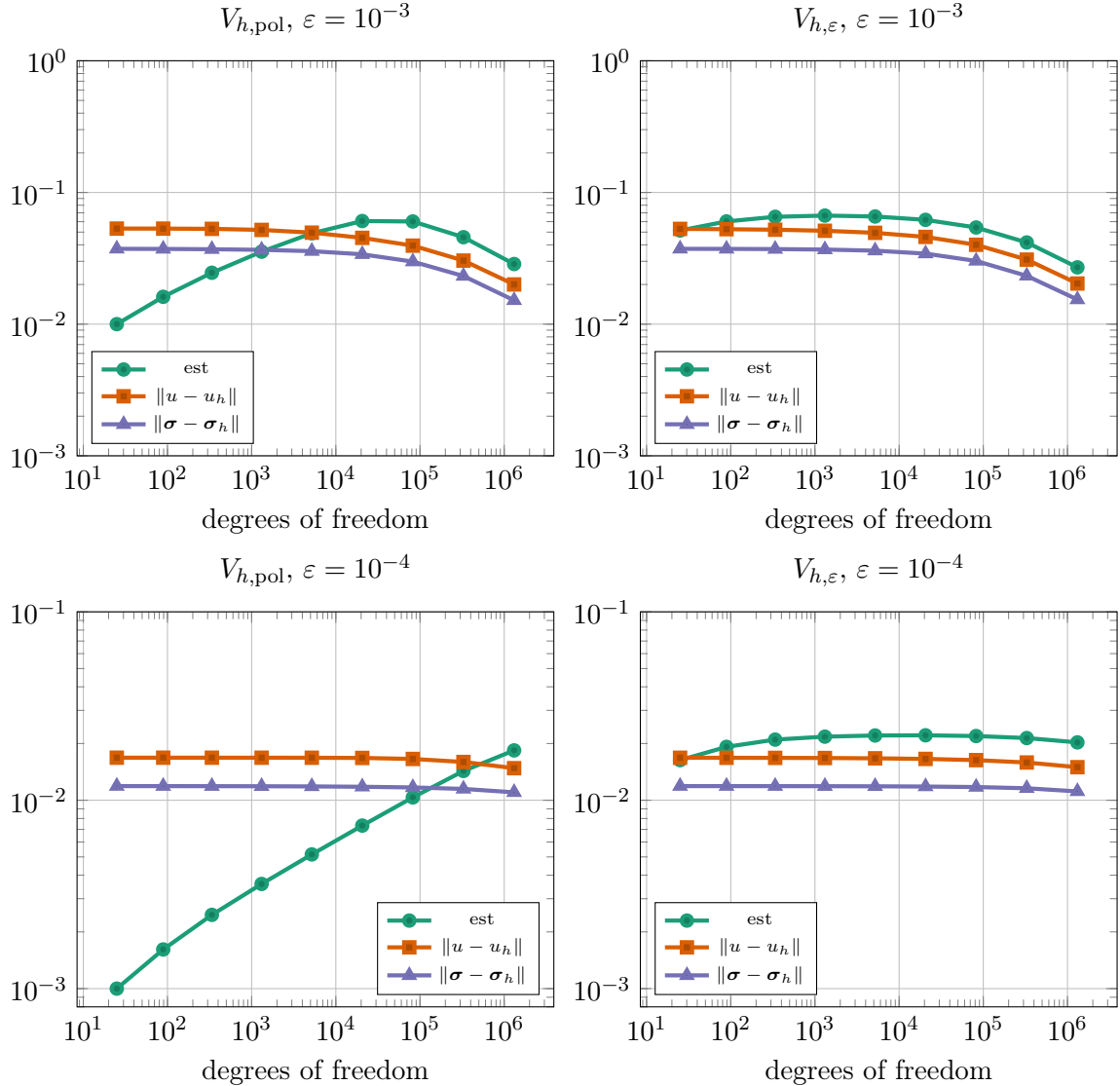

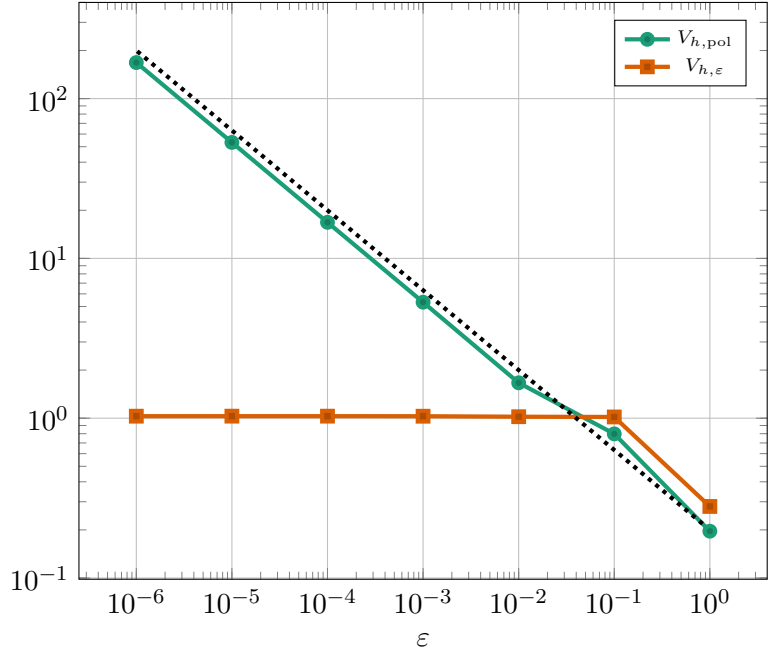
\begin{figure}
  \begin{center}
    \begin{tikzpicture}
\begin{loglogaxis}[
width=0.65\textwidth,
cycle list/Dark2-6,
cycle multiindex* list={
mark list*\nextlist
Dark2-6\nextlist},
every axis plot/.append style={ultra thick},
xlabel={$\varepsilon$},
grid=major,
legend entries={\tiny $V_{h,\mathrm{pol}}$,\tiny $V_{h,\varepsilon}$},
legend pos=north east,
]
\addplot table [x=t,y=ratpol] {data/ratios.dat};
\addplot table [x=t,y=ratmod] {data/ratios.dat};
\addplot [black,dotted,mark=none] table [x=t,y expr={0.2*sqrt(\thisrowno{0})^(-1)}] {data/ratios.dat};

\end{loglogaxis}
\end{tikzpicture}
  \end{center}
  \caption{Ratio $\rho = \norm{u-u_h}{}/\est$ for the two different test spaces $V_{h,\mathrm{pol}}$ and $V_{h,\varepsilon}$ on a fixed mesh. The black dotted line corresponds to $\OO(\varepsilon^{-1/2})$.}
  \label{fig:ratios}
\end{figure}

We use the DPG method from Section~\ref{sec:DPG} with test spaces
\begin{align*}
  V_{h,\mathrm{pol}} = \prod_{T\in\TT} V_{h,\mathrm{pol}}(T), \quad V_{h,\varepsilon} = \prod_{T\in\TT} V_{h,\varepsilon}(T)
\end{align*}
where 
\begin{align*}
  V_{h,\mathrm{pol}}(T) = P^3(T)\times \PP^2(T), \quad 
  V_{h,\varepsilon}(T) = \begin{cases}
    V_{h0}^\nabla \times V_{h}^{\div,2}, & \varepsilon>h_T, \\
    V_{h0,\varepsilon}^\nabla \times V_{h,\varepsilon}^{\div}, & \varepsilon\leq h_T.
  \end{cases}
\end{align*}
The trial space is
\begin{align*}
  U_h = P^0(\TT)\times \PP^0(\TT) \times \trgrad(P^1(\TT)\cap H_0^1(\Omega)) \times \trdiv(\RT^0(\TT)),
\end{align*}
where $\RT^0(\TT) = \set{\ttau\in\LL^2(\Omega)}{\ttau|_T\in \RT^0(T), \, T\in\TT}\cap \Hdivset\Omega$.
We stress that by our constructions from Sections~\ref{sec:fortin} and~\ref{sec:fortinDiv}, $V_h = V_{h,\varepsilon}$ is a test space that allows for a uniformly bounded Fortin operator~\eqref{eq:fortin:abstract}.
On the other hand, test space $V_h = V_{h,\mathrm{pol}}$ allows for a Fortin operator whose norm depends on $\varepsilon$, cf.~\cite{practicalDPG}. 

We also define the DPG error estimator by
\begin{align*}
  \est = \sup_{0\neq \vv_h\in V_h} \frac{b(\uu_h,\vv_h)-L(\vv_h)}{\norm{\vv_h}V}.
\end{align*}
Clearly, this estimator depends on the choice of the test space. 
In~\cite{DPGaposteriori} it is shown that $\est$ is, up to an oscillation term, equivalent to the error 
\begin{align*}
  \norm{\uu-\uu_h}U \eqsim \est + \osc(f) = \est + \sup_{0\neq \vv=(v,\ttau)\in V} \frac{L(\vv-\PiF\vv)}{\norm{\vv}V},
\end{align*}
provided there exists a uniformly bounded Fortin operator $\PiF\colon V\to V_h$.

Figure~\ref{fig:comp} shows the errors of the field variables for $\varepsilon\in \{10^{-3},10^{-4}\}$ and estimator $\est$. 
We observe differences when using $V_h = V_{h,\mathrm{pol}}$ or $V_h=V_{h,\varepsilon}$ as test space:
For coarse meshes and $V_h = V_{h,\mathrm{pol}}$ the estimator $\est$ underestimates the errors in the field variables. This effect is more severe for smaller parameters. 
This can also be seen in Figure~\ref{fig:ratios}. There, we fix a mesh with four elements and only vary $\varepsilon$. We plot the index $\rho = \|u-u_h\|/\est$. When using $V_h = V_{h,\mathrm{pol}}$ we observe that $\rho=\OO(\varepsilon^{-1/2})$ for $\varepsilon\to 0$ whereas $\rho=\OO(1)$ when using $V_h = V_{h,\varepsilon}$. 
We conclude that the DPG method with $V_h = V_{h,\mathrm{pol}}$ is not robust, whereas with the new test space $V_h = V_{h,\varepsilon}$ it is.

\subsection{Discrete stability}\label{sec:stability}
Let $\Omega = \Tref$, $\TT = \{\Tref\}$. In this section we want to study stability of the method by investigating the norm equivalence constants $\lambda_\mathrm{min}$, $\lambda_\mathrm{max}$ in
\begin{align}\label{eq:normequiv}
  \lambda_\mathrm{min} \norm{\uu_h}U^2 \leq b(\uu_h,\Theta_h\uu_h) \leq \lambda_\mathrm{max} \norm{\uu_h}U^2 \quad\forall \uu_h\in U_h.
\end{align}
Here, $U_h$ is defined as in the previous section. We use two different test spaces, $V_{h,\varepsilon}$ (defined in the previous section) and 
\begin{align*}
  \widetilde V_h = V_{h0}^{\nabla}\times V_{h0}^{\div,2}.
\end{align*}
The difficulty in checking the norm equivalence is the implementation of the trace norms. Note that due to inclusion of boundary conditions and the fact that all nodes of $\TT$ are on the boundary, we do not have to consider $\norm{\widehat u_h}{1/2,\varepsilon}$.
To calculate $\norm{\widehat\sigma_h}{-1/2,\varepsilon}$ we generate a submesh $\widetilde\TT$ of $\TT$ such that all elements that have a boundary face have diameter less than or equal to $\varepsilon/2$. This, heuristically, resolves possible boundary layers. To evaluate $\norm{\widehat \sigma_h}{-1/2,\varepsilon}$ we approximate the PDE
\begin{align*}
  \ttau\in\Hdivset{\Omega}\colon \qquad -\varepsilon^2\nabla\div\ttau +\ttau = 0, \quad \ttau\cdot\normal_{\Omega}|_{\partial\Omega} = \widehat\sigma_h
\end{align*}
by a standard FEM on $\widetilde\TT$ using lowest-order Raviart--Thomas elements.
The $\Hdivset\Omega$ norm $\norm{\ttau_h}{\Omega,\varepsilon}$ of the approximation $\ttau_h\in\RT^0(\TT)$ is taken as $\norm{\widehat \sigma_h}{-1/2,\varepsilon}$.

\begin{figure}
  \begin{center}
    \begin{tikzpicture}
\begin{loglogaxis}[
width=0.65\textwidth,
cycle list/Dark2-6,
cycle multiindex* list={
mark list*\nextlist
Dark2-6\nextlist},
every axis plot/.append style={ultra thick},
xlabel={$\varepsilon$},
grid=major,
legend entries={\tiny {$\lambda_\mathrm{max}$, $\widetilde V_{h}$},\tiny {$\lambda_\mathrm{min}$, $\widetilde V_{h}$},\tiny {$\lambda_\mathrm{max}$, $V_{h,\varepsilon}$},\tiny {$\lambda_\mathrm{min}$, $V_{h,\varepsilon}$}},
legend pos=south east,
]
\addplot table [x=t,y=ellmax] {data/eigval.dat};
\addplot table [x=t,y=ellmin] {data/eigval.dat};
\addplot table [x=t,y=ellmaxMB] {data/eigval.dat};
\addplot table [x=t,y=ellminMB] {data/eigval.dat};
\addplot [black,dotted,mark=none] table [x=t,y expr={4*(\thisrowno{0})^(1)}] {data/eigval.dat};

\end{loglogaxis}
\end{tikzpicture}
  \end{center}
  \caption{Constants $\lambda_\mathrm{max}$ and $\lambda_\mathrm{min}$ from~\eqref{eq:normequiv} for the test spaces $\widetilde V_h$ and $V_{h,\varepsilon}$ (see Section~\ref{sec:stability}). The black dotted line corresponds to $\OO(\varepsilon^{-1})$.}
  \label{fig:normequiv}
\end{figure}
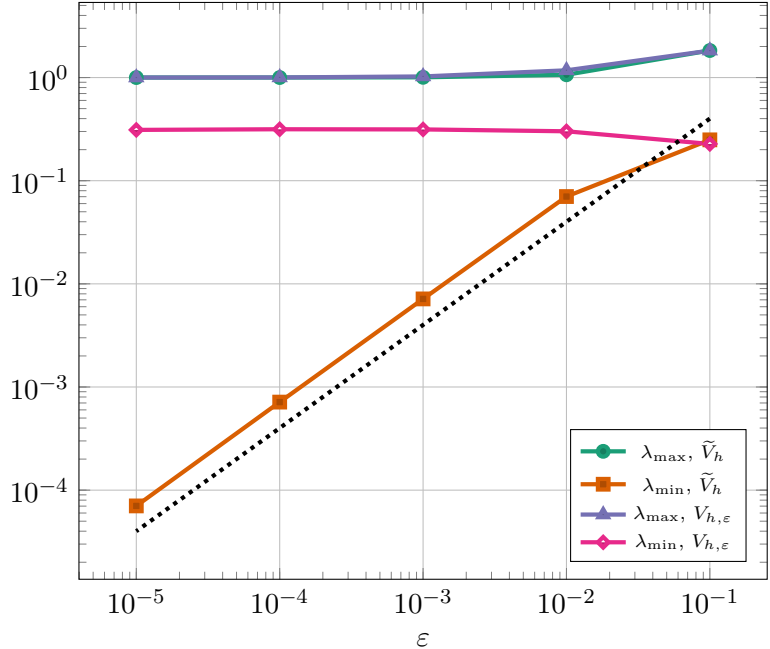

In view of norm equivalence~\eqref{eq:normequiv} we stress that $\lambda_\mathrm{max}\lesssim 1$ independent of $\varepsilon$ and test space $V_h$ since the bilinear form is uniformly bounded. However, the lower bound is directly related to the stability of the DPG method, i.e., $\lambda_\mathrm{min}$ depends on the discrete $\inf$-$\sup$ constant which for the DPG method is related to the norm of Fortin operators as we already discussed in the introduction.
Figure~\ref{fig:normequiv} visualizes $\lambda_\mathrm{min}$ and $\lambda_\mathrm{max}$ for $V_h = V_{h,\varepsilon}$ and $V_h = \widetilde V_h$.
We observe that $\lambda_\mathrm{max}$ is uniformly bounded for both test spaces (for small $\varepsilon$ we can not even distinguish them in the plot), whereas $\lambda_\mathrm{min}$ deteriorates for $V_h = \widetilde V_h$ and is essentially constant for $V_h=V_{h,\varepsilon}$.
This illustrates that the DPG method with test space $V_{h,\varepsilon}$ is uniformly stable, in contrast to the canonical method with test space $\widetilde V_h$.

\bibliographystyle{abbrv}
\bibliography{literature}

\end{document}